\documentclass[10pt, reqno]{amsart}

\usepackage{amsmath,amssymb,amsthm}
\usepackage{trfsigns}
\usepackage{amsmath, amssymb}
\usepackage{amsthm, amsfonts, mathrsfs}
\usepackage{fullpage}
\usepackage{amsfonts,graphicx}
\usepackage[english]{babel} 
\usepackage{blindtext}
\numberwithin{equation}{section}
\usepackage[colorlinks=true, pdfstartview=FitV, linkcolor=blue, citecolor=blue, urlcolor=blue]{hyperref}

\newtheorem{Theo}{Theorem}[section]
\newtheorem{lem}[Theo]{Lemma}

\newtheorem{prop}[Theo]{Proposition}

\theoremstyle{plain}
\theoremstyle{definition}
\newtheorem{defi}[Theo]{Definition}

\theoremstyle{remark}

\newtheorem*{rema*}{Remark}

\newcommand{\NN}{\mathbb{N}}

\newcommand{\RR}{\mathbb{R}}

\newcommand{\Div}{\textnormal{div}}

\newcommand{\DD}{\textnormal{D}}

\newcommand{\supp}{\textnormal{supp }}

\author[A.Rahmani]{Anis Rahmani}
\address{Former-Student of Department of Mathematics, University of Batna 2, Mostefa Ben
Boula\"id, Fesdis, Batna 05078, Algeria.}
\email{rahmanianis58@gmail.com}

\author[A. Mennouni]{Abdelaziz Mennouni}
\address{Department of Mathematics, LTM, University of Batna 2, Mostefa Ben Boula\"id, Fesdis, Batna 05078, Algeria.}
\email{a.mennouni@univ-batna2.dz}

\keywords{Fractional Magnetic B\'enard,  MHD system,  Boussinesq system, Global well-posedness, Besov space}
\subjclass[2000]{35K15, 35K55, 35K65,35B40.}

\begin{document}

\title{New approach in the Besov space to the existence and uniqueness of solutions of the D-dimensional fractional magnetic Bénard system without thermal diffusion}

\date{}
\maketitle
\begin{abstract}
This work investigates the existence and uniqueness of local weak solutions for the d-dimensional $(d \geq 2)$ fractional magnetic B\'enard system without thermal diffusion, integrating the Bénard equation and MHD system. For $\kappa = 0$ and $1 \leq \alpha=\beta < 1 + \dfrac{d}{4}$, we establish that any starting conditions $(u_{0},B_{0})\in  B_{2,1}^{1+\frac{d}{2}-2\alpha}(\mathbb{R}^{d})$ and $\theta_{0}\in  B_{2,1}^{1+\frac{d}{2}-\alpha}(\mathbb{R}^{d})$.
\end{abstract}
\tableofcontents
\section{Introduction}
\par Nonlinear partial differential equations are used to explain scientific contexts in several domains (e.g., \cite{BTM,BTM23,BTM24,OM24}).
\par Driven by the discussions and outcomes mentioned above, we aim to showcase the reality and distinctiveness of the subsequent d-dimensional fractional magnetic Bénard system characterized by the absence of thermal diffusion, i.e., $\kappa =0$:
The Bénard system with magnetic properties The article cited, Mte, expands upon the traditional Bénard convection phenomenon by introducing the influence of a magnetic field. Classic Bénard convection occurs when a fluid layer is heated from below and chilled from above, resulting in a temperature gradient. This gradient induces convection, which is determined by the Rayleigh and Prandtl numbers. In addition, the magnetic field adds intricacy by affecting the convective patterns through the Lorentz force exerted on the electrically conductive fluid. There are many other natural and engineering environments where this paradigm can be used. It is useful in understanding Earth's mantle convection and other processes related to geophysical fluid dynamics, see \cite{Ped}. The study of astrophysics sheds light on internal processes in stars. The model can also be applied in several industrial contexts where magnetohydrodynamic (MHD) flows are crucial, like liquid metal cooling systems and fusion reactors for maintaining the magnets inside.
\par The focus of this study is on determining whether local weak solutions to a d-dimensional magnetic B\'enard model, which may be expressed as\begin{equation}\label{sys1}
\left\{
\begin{aligned}
    & \partial_{t} u + u\cdot\nabla u  =- \mu (-\Delta)^{\alpha}u-\nabla p + B\cdot \nabla B +\theta e_{d},\quad (x,t)\in \mathbb{R}^{d}\times (0,\infty),
     \\
    & \partial_{t} B + u\cdot\nabla B =- \eta (-\Delta)^{\beta}B + B\cdot \nabla u,\quad (x,t)\in \mathbb{R}^{d}\times (0,\infty),
     \\
   & \partial_{t}\theta + u\cdot\nabla\theta =-\kappa(-\Delta)^{\gamma}\theta + u\cdot e_{d},\quad (x,t)\in \mathbb{R}^{d}\times (0,\infty),
  \\
&\Div{u}=0 ~~~~~\Div{B}=0,\quad (x,t)\in \mathbb{R}^{d}\times (0,\infty),
\\
& (u,B,\theta )\mid_{t=0} = (u_{0}, B_{0},\theta _{0}),\quad (x,t)\in \mathbb{R}^{d}\times (0,\infty).
\end{aligned}
\right.
\end{equation}
where $\eta \geq 0 $ represents magnetic diffusion and $\mu \geq 0$ represents kinematic viscosity. The variables $u(x,t),\;B(x,t),
\\\theta(x,t)$ and $p(x,t)$ signify the velocity field, the magnetic field, the temperature, and the pressure, respectively. Additionally, the unit vector in the vertical direction is $e_{d}=(0,0,...,1)$; the acting buoyancy force on the fluid motion is reported by $\theta e_{d}$; and the Rayleigh–Bénard convection in a heated inviscid fluid is simulated by $u\cdot e_{d} $. The Fourier transform defines the fractional Laplacian operator $(-\Delta)^{\alpha} $ as follows:
$$
\widehat{(-\Delta)^{\alpha}} f(\xi)=|\xi|^{2\alpha}\hat{f}(\xi)
$$
and 
$$
\hat{f}(\xi)=\frac{1}{(2\pi)^{\frac{d}{2}}}\int_{\mathbb{R}^{d}}e^{-ix\cdot \xi} f(x)dx.
$$
The conventional magnetic B\'enard system reduces to the B\'enard equation if $B=0$. The Boussinesq equation, which models geophysical phenomena including air fronts and oceanic currents, is created when the Rayleigh-B\'enard convection term $u\cdot e_{d}=0$. This equation is crucial for understanding Rayleigh-B\'enard convection. The Boussinesq equations and the B\'enard equation have been the subject of much research and study due to their significance in mathematics and physics, (see, e.g., \cite{b,CDib,DP,dp,Guo,HKR2,HZ,Tit}).
Once $\theta = 0$, the temperature impact is ignored, transforming the conventional Magnetic B\'enard system into the MHD system. For readers interested in a wealth of fascinating results on the MHD system, please see \cite{YJ,hk,HR}.
The global regularity of solutions is a fundamental issue with the Magnetic Bénard system. When $\alpha=\beta=\gamma=1$, the conventional Magnetic B\'enard system is represented as (\ref{sys1}).
\par There have been advancements in $2D$ and $3D$ magnetic B\'enard systems in recent years. According to \cite{Hp}, solutions retain regularity if the magnetic field has a full Laplacian with appropriate fractional Laplacians for temperature and velocity, if there is no diffusion in the magnetic and temperature fields with logarithmically supercritical dissipation for velocity, or if there is zero dissipation for velocity with a full Laplacian for temperature and fractional Laplacian for the magnetic field.
The paper \cite{BC} looks into the regularity of weak solutions to the 3D magnetic Bénard system. It does this by turning the system into a symmetric form and getting better results with enough integrable regularity conditions for pressure and gradient pressure in Besov spaces that are all the same.
In \cite{Dr}, the author investigated the global regularity problem for the $2D$ magnetic B\'enard problem with fractional dissipation, in which the dissipation terms are $(-\Delta)^{\alpha}u$, $(-\Delta)^{\beta}u$ and $(-\Delta)^{\gamma}\theta$. The author proved that smooth solutions are global in time for the following three important cases, namely $0<\alpha <\frac{1}{2},\beta =1$ and $\gamma =\frac{1}{2}$ or $\alpha =0 ,\beta >1 $ and $\gamma >\frac{1}{2}$ or $\alpha =\frac{1}{2},\beta =1$
and $\gamma =0$. The known worldwide regularity results are much enhanced by these studies. The following articles \cite{che1,che2,che3,HL,ES,MAL} are also relevant to the regularity topic.
\par Encouraged by the previous findings and discussions, we will present the existence and uniqueness for the d-dimensional fractional magnetic B\'enard system (\ref{sys1}) with zero thermal diffusion ie. $\kappa =0$.
\begin{Theo}\label{Th1}
    Let $d\geq 2 $. Consider the system  (\ref{sys1}) with $\kappa =0,\; \mu > 0,\; \eta >0$ and $ 1 \leq \alpha=\beta < 1 +\dfrac{d}{4}$. Assume the initial data $(u_{0},B_{0},\theta_{0})$ satisfy
\begin{equation*}
    \begin{aligned}
        (u_{0},B_{0})&\in  B_{2,1}^{1+\frac{d}{2}-2\alpha}(\mathbb{R}^{d}),\;\;\;\; \theta_{0}\in  B_{2,1}^{1+\frac{d}{2}-\alpha}(\mathbb{R}^{d})\quad\mbox{and} \quad\Div{(u_{0})}=\Div{(B_{0})}=0.
    \end{aligned}
\end{equation*}
Then there exists a small enough time $T $ such that the system (\ref{sys1}) has a unique weak solution $(u,B,\theta)$ on $[0,T ]$ satisfying
\begin{equation*}
    \begin{aligned}
  u&\in L^{\infty}(0,T; B_{2,1}^{1+\frac{d}{2}-2\alpha}(\mathbb{R}^{d}))\cap L^{1}(0,T; B_{2,1}^{1+\frac{d}{2}}(\mathbb{R}^{d})),\\
 B&\in L^{\infty}(0,T; B_{2,1}^{1+\frac{d}{2}-2\alpha}(\mathbb{R}^{d}))\cap L^{1}(0,T; B_{2,1}^{1+\frac{d}{2}}(\mathbb{R}^{d})),\\
\theta &\in L^{\infty}(0,T; B_{2,1}^{1+\frac{d}{2}-\alpha}(\mathbb{R}^{d})).
    \end{aligned}
\end{equation*}
\end{Theo}
\section{Dimensionless Parameters}
\textbf{Rayleigh Number (Ra):}
A dimensionless number that quantifies the buoyancy-driven flow about viscous damping: it controls the start of convection. The existence of a magnetic field can affect the critical Rayleigh number of the magnetic Bénard system.
\begin{eqnarray*}
    \textbf{Ra}=\dfrac{g \beta \Delta T L^{3}}{\nu \kappa}.
\end{eqnarray*}
where $\Delta T $ is the temperature difference across the fluid layer, $L$ is the characteristic length (e.g., the height of the fluid layer), 
$\nu$ is the kinematic viscosity, and $\kappa $ is the thermal diffusivity. 
\\
\textbf{Prandtl Number (Pr):}
\begin{eqnarray*}
    Pr=\dfrac{\nu}{\kappa}.
\end{eqnarray*}
It describes the thickness of the velocity boundary layer in relation to the thermal boundary layer.
\\
\textbf{Hartmann Number (Ha):}
\begin{eqnarray*}
    Ha=\dfrac{B_{0}L \sqrt{\sigma}}{\sqrt{\nu \mu}}
\end{eqnarray*}
where $B_{0}$ is the magnetic field strength, and 
$\sigma $ is the electrical conductivity of the fluid. The Hartmann number measures the relative importance of magnetic forces compared to viscous forces.
\\
\textbf{Magnetic Reynolds Number (Rm):}
\begin{eqnarray*}
    R_{m}=\dfrac{UL}{\eta}.
\end{eqnarray*}
where $U$ is a characteristic velocity, and $\eta $ is the magnetic diffusivity. It measures the relative importance of advection of the magnetic field by the fluid flow to diffusion of the magnetic field.
\begin{rema*}\textbf{Patterns of Stability and Convection} \\
Both the stability and the onset of convection are influenced by the interaction of these dimensionless factors. The critical Rayleigh number $R_{{a}_{c}}$ establishes the beginning of convection. A magnetic field raises $R_{{a}_{c}}$ and stabilizes the flow, as indicated by the Hartmann number. Convection patterns are also influenced by the magnetic field; their strength and orientation determine whether they become more organized or suppressed.

\end{rema*}
\section{Background and tool lemmas}
This part will review the definition of Besov space, some important lemma, and some basic information on the Littlewood-Paley theory.
\begin{defi}[See \cite{GL}]
    Let $\mathcal{S}(R^{d})$ be the space of Schwartz class of rapidly decreasing
functions such that for any $k\in \mathbb{N}$
\begin{eqnarray*}
    \lVert u \rVert_{k,\mathcal{S}}=\sup_{\lvert \alpha \rvert\leq k , x\in \mathbb{R}^{d} } (1+\lvert x \rvert)^{k} \lvert \partial^{\alpha} u(x)\rvert \leq \infty.
\end{eqnarray*}
\end{defi}
\begin{rema*} It is crucial to remember that:
\begin{enumerate}
      \item  The set $\mathcal{S'}$ of temperate distributions is the dual set of $\mathcal{S}$ for the usual pairing.
    \item  the Zygmund operator $\Lambda$, also known as the square root of the Laplacian i.e. related to the fractional Laplacian $(-\Delta)^{\frac{\alpha}{2}}$. Specifically, for $\alpha=1$, we have 
$(-\Delta)^{\frac{1}{2}}=\Lambda$. 
\end{enumerate}
\end{rema*}
Let  $1 \leq p \leq \infty$, denote by $q$ the conjugate of $p$, that is to say
     $$
		\dfrac{1}{p}+\dfrac{1}{q}=1.
		$$
\begin{lem}[See \cite{BM}]
 Assume that $f\in L^{p}(\mathbb{R}^{n})$ and $g\in L^{q}(\mathbb{R}^{n})$ with $ 1 \leq p \leq \infty $. Then 
    \begin{equation*}
        \begin{aligned}
  fg\in L^{1} (\mathbb{R}^{n})\;\; \mbox{and}\;\; \int\limits_{\mathbb{R}^{n}}\lvert f g \rvert dx \leq \lVert f\rVert_{L^{p}(\mathbb{R}^{n})} \lVert g \rVert_{L^{q}(\mathbb{R}^{n})}.
  \end{aligned}
    \end{equation*}
\end{lem}
 \begin{defi}[See \cite{BM}]
Assume that $f \in L^{p}(\mathbb{R}^{n})$ and $g \in L^{q}(\mathbb{R}^{n})$ with
$$
1 \leq p \leq \infty,\;\;1\leq q \leq \infty
    \;\; \mbox{and}\;\; \dfrac{1}{r}=\dfrac{1}{p}+\frac{1}{q}-1 \geq 0.
		$$
  Then
	$$
 f\ast g\in L^{r}(\mathbb{R}^{n}) \;\;\mbox{and}\;\;\lVert f\ast g \rVert_{L^{r}} \leq \lVert f\rVert_{L^{p}(\mathbb{R}^{n})} \lVert g \rVert_{L^{q}(\mathbb{R}^{n})}.
 $$ 
	\end{defi}
\subsection{Littlewood-Paley theory}
The Littlewood-Paley theorem was proposed by John Edensor Littlewood and Raymmond Paley in the 1930s. John-Michel Bony's 1981 paper on the para-differential calculus, which connects the nonlinear function and the Littlewood-Paley composition, marked a significant breakthrough in the systematic use of this theorem in the analysis of partial differential equations.
\par Let $\chi\in\mathscr{D}(\RR^d)$ be a reference cut-off function, monotonically decaying along rays and so that
\begin{eqnarray*}
\begin{cases}
  \chi\equiv1       & \text{if } \lVert\xi\rVert\le\frac12, \\
	0\le \chi\le1  & \text{if } \frac12\le\lVert\xi\rVert\le1, \\
  \chi\equiv0        & \lVert\xi\rVert\ge1.
  \end{cases}
\end{eqnarray*}
Define $\varphi(\xi)\triangleq\chi(\frac{\xi}{d})-\chi(\xi)$. Note that $\varphi\ge0$ and $$\supp\varphi\subset\mathcal{C}\triangleq\{\xi\in\RR^d:\frac12\le\|\xi\|\le1\}.$$  
Next, we have the following fundamental characteristics: For instance, refer to \cite{HCD,HKR1}.
\begin{prop} Let $\chi$ and $\varphi$ be as above. Then the following assertions are hold.
\begin{enumerate}
\item[(1)] Decomposition of the unity: 
$$
\chi(\xi)+\sum_{q\ge0}\varphi(2^{-k}\xi)=1,\;\;\mbox{for all}\;\;\xi\in\RR^d;
$$
\item[(2)] Almost orthogonality in the sense of $L^2$:
$$
\frac{1}{2}\le\chi^2(\xi)+\sum_{k\ge0}\varphi^2(2^{-k}\xi)\le1,\;\;\mbox{for all}\;\;\xi\in\RR^d.
$$
\end{enumerate}
\end{prop}
Let us review the Littlewood-Paley operator denoted by $\Delta_{k}$.
\begin{defi} For every $u\in\mathcal{S}'(\RR^d)$, setting 
\begin{equation*}
\Delta_{-1}u\triangleq\chi(\DD)u,\;\; \Delta_{k}u\triangleq\varphi(2^{-k}\DD)u\;\;\mbox{if}\;\;k\in\NN,\;\; S_{k}u\triangleq\sum_{j\le k-1}\Delta_{j}u\;\;\mbox{for}\; k\ge 0.
\end{equation*}
\end{defi}
Some properties of $\Delta_q$ and $S_q$ are listed in the following proposition.
\begin{prop} Let $u,v\in\mathcal{S}'(\RR^d)$ we have
\begin{enumerate}
	\item If $\vert p-k\vert\ge2$. Then $\Delta_p\Delta_k u\equiv0$;
	\item If $\vert p-k\vert\ge4$. Then $\Delta_q(S_{p-1}u\Delta_p v)\equiv0$;
		\item The operators $\Delta_k, S_k: L^p\rightarrow L^p$ are uniformly bounded with respect to  $k$ and $p$.
\end{enumerate}
 \end{prop}
Likewise the homogeneous operators $\dot{\Delta}_{k}$ and $\dot{S}_{k}$ are defined by
\begin{equation}\label{Hom}
\forall{k}\in \mathbb{Z}\quad\dot{\Delta}_{k}=\varphi(2^{k}D)u, \quad \dot{S}_{k}=\sum_{ j\le k-1}\dot{\Delta}_{j}v.
\end{equation}
We shall now review the Besov spaces' definition. To this end, denoted by ${\bf P}$ the set of all polynomials.
\begin{defi} For $(s,p,r)\in\RR\times[1,  +\infty]^2$. The inhomogeneous Besov space $B_{p,r}^s$ (resp. the homogeneous Besov space $\dot{B}_{p,r}^s$) is the set of all tempered distributions $u\in\mathcal{S}^{'}$ (resp. $u\in\mathcal{S}^{'}_{|{\bf P}})$ such that
\begin{equation*}
    \begin{aligned}
        &\Vert u\Vert_{{B}_{p, r}^{s}}\triangleq\Big(2^{qs}\Vert \Delta_{q} u\Vert_{L^{p}}\Big)_{\ell^r}<\infty. \\
&\mbox{resp. }\\
&\Vert u\Vert_{\dot{{B}}_{p, r}^{s}}\triangleq\ \Big(2^{qs}\Vert \dot\Delta_{q} u\Vert_{L^{p}}\Big)_{\ell^r(\mathbb{Z})}<\infty. 
    \end{aligned}
\end{equation*}
\end{defi}
\subsection{Paradifferential calculus}
We can formally divide the product of two tempered distributions $u$ and $v$ into three parts using the well-known {\it Bony's} decomposition \cite{b11}. We will remainder the following definitions:
\begin{defi}
 For a given $u, v\in\mathcal{S}'$ we have the following Bony decomposition:
\begin{eqnarray*}
    uv=T_u v+T_v u+\mathscr{R}(u,v),
\end{eqnarray*}
where
\begin{equation*}
\begin{aligned}
T_u v &=\sum\limits_{k \geq -1} S_{k-1}u \Delta_{k} v,\\
T_v u &=\sum\limits_{k \geq -1} S_{k-1} v\Delta_{k} u,\\
\mathscr{R}(u,v) &=\sum\limits_{k \geq -1} \Delta_{k}u \tilde{\Delta}_{k} v
\end{aligned}
\end{equation*}
and
\begin{eqnarray*}
    \tilde{\Delta}_{k}= \Delta_{k-1}+\Delta_{k}+\Delta_{k+1}.
\end{eqnarray*}
\end{defi}
Further details about the dyadic block $\Delta_{k}$ and the Littlewood-Paley theory are provided in \cite{HCD,HKR1,Lk}.
\par We represent the following Besov spaces of Chemin-Lerner type, initially presented in \cite{HCD}, as the mixed space-time spaces.
\begin{defi}\label{def1}
Let $T>0$ and $(r,p,q,s)\in[1, \infty]^3\times\RR$.  We define the spaces $L^{r}_{T}B_{p,q}^s$ and $\widetilde L^{r}_{T}B_{p,q}^s$ respectively by: 
\begin{equation*}
\begin{aligned}
     & L^{r} (0,T;B_{p,q}^{s})=L^r_{T}B_{p,q}^s\triangleq\Big\{u: [0,T]\to\mathcal{S}^{'}; \Vert u\Vert_{L_{T}^{r}B_{p, q}^{s}}=\big\Vert\big(2^{ks}\Vert \Delta_{k}u\Vert_{L^{p}}\big)_{\ell^{q}}\big\Vert_{L_{T}^{r}}<\infty\Big\},
     \\
     &\tilde{L}^{r}(0,T;B_{p,q}^{s})=\widetilde L^{r}_{T}B_{p,q}^s\triangleq\Big\{u:[0,T]\to\mathcal{S}^{'}; \Vert u\Vert_{\widetilde L_{T}^{r}{B}_{p, q}^{s}}=\big(2^{ks}\Vert \Delta_{k}u\Vert_{L_{T}^{r}L^{p}}\big)_{\ell^{q}}<\infty\Big\}.
\end{aligned}
\end{equation*}
Based on Minkowski's inequality, the following embeddings show the connection between these spaces:
\begin{eqnarray*}
\nonumber\tilde{L}^{r}(0,T;B_{p,q}^{s})\subsetneq L^{r}(0,T;B_{p,q}^{s}), \;\;\mbox{ if}\;\; r>q,\\
\nonumber\tilde{L}^{r}(0,T;B_{p,q}^{s})\supsetneq L^{r}(0,T;B_{p,q}^{s}), \;\;\mbox{ if}\;\; r<q,\\
\nonumber\tilde{L}^{r}(0,T;B_{p,q}^{s})=L^{r}(0,T;B_{p,q}^{s}), \;\;\mbox{ if}\;\; r=q.
\end{eqnarray*}
\end{defi}
Assuming that the Fourier transform of the function is compactly supported, the following {\it Bernstein} inequalities characterize a bound on the derivatives of a function in the $L^b-$norm in terms of the value of the function in the $L^a-$norm. For more details, see \cite{HCD,che10}.
\begin{lem}\label{lem1}
    Let $\alpha \geq 0 $ and $1 \leq p \leq q \leq \infty $ \\
 (1) If $f$ satisfies
 \begin{equation*}
      \supp\widehat{f} \subset \lbrace \xi \in \mathbb{R}^{d} : | \xi | \leq K 2^{j} \rbrace\;\;\mbox{for some integer $j$ and a constant $ K > 0$}.
 \end{equation*}  
 Then
 \begin{equation*}
       \lVert (-\Delta)^{\alpha} f \rVert_{L^{q}(\mathbb{R}^{d})} \leq C_{1} 2^{2\alpha j+jd(\frac{1}{p}-\frac{1}{q})}\lVert  f \rVert_{L^{p}(\mathbb{R}^{d})}
 \end{equation*}
 (2) If $f$ satisfies
  \begin{equation*}
        \supp \widehat{f} \subset \lbrace \xi \in \mathbb{R}^{d} : K_{1} 2^{j} \leq | \xi | \leq K_{2} 2^{j} \rbrace,\;\;\mbox{for some integer $j$ and constants $ 0 < K_{1}\leq K_{2}$}.
 \end{equation*}
  Then
  \begin{equation*}
C_{1} 2^{2\alpha j}\lVert  f \rVert_{L^{p}(\mathbb{R}^{d})} \leq    \lVert (-\Delta)^{\alpha} f \rVert_{L^{q}(\mathbb{R}^{d})} \leq C_{2} 2^{2\alpha j+jd(\frac{1}{p}-\frac{1}{q})}\lVert  f \rVert_{L^{p}(\mathbb{R}^{d})} 
 \end{equation*}
 where $C_{1}$ and $C_{2}$ are constants depending on $\alpha$, $p$ and $q$.
\end{lem}
\begin{lem}\label{lem2}
    Let $j\in\mathbb{Z}$ be an integer. Let $\Delta_{j}$ be a dyadic block operator (either inhomogeneous or homogeneous). For any vectors field $u, v, w$ with $\nabla\cdot u=0$, we have
    \begin{equation*}
        \begin{aligned}
|\int_{\mathbb{R}^{d}}\Delta_{j}(v\cdot\nabla u)\cdot\Delta_{j}w d x|
&\leq C\|\Delta_{j}w\|_{L^{2}}(2^{j}\sum\limits_{m\leq j-1}2^{\frac{d}{2}m}\|\Delta_{m}v\|_{L^{2}}\sum\limits_{|j-k|\leq2}\|\Delta_{k}u\|_{L^{2}}\\
&+\sum\limits_{|j-k|\leq2}\|\Delta_{k}v\|_{L^{2}}\sum\limits_{m\leq j-1}2^{(1+\frac{d}{2})m}\|\Delta_{m}u\|_{L^{2}}\\
&+\sum\limits_{k\geq j-4}2^{j}2^{\frac{d}{2}k}\|\Delta_{k}v\|_{L^{2}}\|\tilde{\Delta}_{k}u\|_{L^{2}}),
        \end{aligned}
    \end{equation*}
    \begin{equation*}
       \begin{aligned}
  |\int_{\mathbb{R}^{d}}\Delta_{j}(u\cdot\nabla v)\cdot\Delta_{j}v dx|
 &\leq C\|\Delta_{j}v\|_{L^{2}}(\sum\limits_{m\leq j-1}2^{(1+\frac{d}{2})m}\|\Delta_{m}u\|_{L^{2}}\sum\limits_{|j-k|\leq2}\|\Delta_{k}v\|_{L^{2}}\\
&+\sum\limits_{|j-k|\leq2}\|\Delta_{k}u\|_{L^{2}}\sum\limits_{m\leq j-1}2^{(1+\frac{d}{2})m}\|\Delta_{m}v\|_{L^{2}}\\
&+\sum\limits_{k\geq j-4}2^{j}2^{\frac{d}{2}k}\|\Delta_{k}u\|_{L^{2}}\|\tilde{\Delta}_{k}v\|_{L^{2}})
\end{aligned}
\end{equation*}
\begin{equation*}
    \begin{aligned}
        |\int_{\mathbb{R}^{d}}\Delta_{j}(vw)\cdot\Delta_{j}u dx|&\leq C\|\Delta_{j}u\|_{L^{2}}(\sum\limits_{m\leq j-1}2^{\frac{d}{2}m}\|\Delta_{m}v\|_{L_{2}}\sum\limits_{|j-k|\leq2}\|\Delta_{k}w\|_{L^{2}}\\
&+\sum\limits_{|j-k|\leq2}\|\Delta_{k}v\|_{L^{2}}\sum\limits_{m\leq j-1}2^{\frac{d}{2}m}\|\Delta_{m}w\|_{L^{2}}\\
&+\sum\limits_{k\geq j-4}2^{\frac{d}{2}k}\|\Delta_{k}v\|_{L^{2}}\|\tilde{\Delta}_{k}w\|_{L^{2}}).
    \end{aligned}
\end{equation*}
\end{lem}
Here we provide the Aubin-Lions-Simon Lemma given in \cite[Theorem 6]{che0}.
\begin{lem}
   Let $X$, $B$ and $Y$ be Banach spaces satisfying $X\subset B \subset Y$ with compact imbedding $X \subset\subset B $. If $f_{n}$ is bounded in $L^{q}(0,T;B)$, $1<q \leq \infty$ and $L_{loc}^{1}(0,T;X)$) and if $\dfrac{\partial f_{n}}{\partial_{t}}$ is 
bounded in $L_{loc}^{1}(0,T;Y)$, then $f_{n}$ is relatively compact in $L^{p}(0,T;B)$ for any $p<q$.
\end{lem}
\begin{lem}\label{lem4} (Grönwall’s inequality)
Let $0<a<1$, $f$ be a measurable function, $\phi$ a locally integrable function and $\varphi$ a positive, continuous and nondecreasing function. Assume that for some nonnegative real number $c$, the function $f$ satisfies
\begin{equation*}
        f(t)\leq c+\int_{t_{0}}^{t} \phi(\tau)\varphi(f(\tau)) d\tau\;\;\mbox{for some positive constant}\;\;c.
\end{equation*}
Then 
\begin{equation*}
\psi(f(t))+\psi(c)\leq\int_{t_{0}}^{t}\phi(\tau) d\tau\;\;\mbox{and}\;\; \psi(x)=\int_{x}^{a}\frac{
dr}{\varphi(r)}.
\end{equation*}
If $c=0$ and $\varphi$ satisfies
\begin{equation*}
    \int_{0}^{a}\frac{dr}{\varphi(r)}=\infty.
\end{equation*}
Then  $f\equiv0$.
\end{lem}
\section{Proof for the existence part of Theorem 1.1}
We concentrate on establishing Theorem 1.1 in this section. Building successive approximation sequences and demonstrating that a subsequence's limit effectively solves (1.1) in the weak sense are our main objectives. To demonstrate existence and uniqueness, we employ alternative models, such as those found in \cite{Fa,GP}.
\par Let us consider a successive approximation sequence $(u^{(n)},B^{(n)},\theta^{(n)})$ that satisfies
\begin{equation}\label{sys2}
    \left\{
    \begin{aligned}
         & u^{(1)}=S_{1}u_{0}, \quad B^{(1)}=S_{1}B_{0}, \quad \theta^{(1)}=S_{1}\theta_{0},
        \\
        &\partial_{t}u^{(n+1)}+\mu(-\Delta)^{\alpha}u^{(n+1)}= \mathbb{P}\bigg[(-u^{(n)}\cdot\nabla u ^{(n+1)}  + B^{(n)}\cdot \nabla B^{(n)}+\theta^{(n)}e_{d}\bigg],
\\
&\partial_{t} B^{(n+1)}+\eta(-\Delta)^{\alpha}B^{(n+1)}=-u^{(n)}\cdot\nabla B^{(n+1)}+B^{(n)}\cdot \nabla u^{(n)},\\
&\partial_{t}\theta^{(n+1)}=-u^{(n)}\cdot\nabla \theta^{(n+1)}+u^{(n)}\cdot e_{d},\\
&\Div{u^{(n+1)}}=\Div{B^{(n+1)}}=0,\\
&u^{(n+1)}(x,0)=S_{n+1}u_{0},\\
& B^{(n+1)}(x,0)=S_{n+1}B_{0},\\
&\theta^{(n+1)}(x,0)=S_{n+1}\theta_{0}.
    \end{aligned}
    \right.
\end{equation}
where $S_{n}$ is the standard inhomogeneous low frequency cutoff operator and $\mathbb{P}=I-\nabla(-\Delta)^{-1}\Div$ is the standard Leray projection.
\par  Set
\begin{align*} 
M&:=2\big{(}\|u_{0}\|_{B_{2,1}^{1+\frac{d}{2}-2\alpha}}+\|B_{0}\|_{B_{2,1}^{1+
\frac{d}{2}-2\alpha}}+\|\theta_{0}\|_{B_{2,1}^{1+\frac{d}{2}-\alpha}}\big{)},\\
Y&:=\Big{\{}(u,B,\theta) \ \Big{|}\ \nonumber\|u\|_{\tilde{L}^{\infty}(0,T;B_{2,1}^{1+\frac{d}{2}-2\alpha})}\leq M, \ \|B\|_{\tilde{L}^{\infty}(0,T;B_{2,1}^{1+\frac{d}{2}-2\alpha})}\leq M, \\ &\|\theta\|_{\tilde{L}^{\infty}(0,T;B_{2,1}^{1+\frac{d}{2}-\alpha})}\leq M, \ \|u\|_{L^{1}(0,T;B_{2,1}^{1+\frac{d}{2}})}\leq \delta, \ \|B\|_{L^{1}(0,T;B_{2,1}^{1+\frac{d}{2}})}\leq \delta \Big{\}}.
\end{align*}
For some small $T > 0$ and $0 <\delta < 1$.
\par In the following Lemma we will prove an estimate for $\lVert u^{(n+1)} \rVert_{\tilde{L}^{\infty}\big((0,T);B_{2,1}^{1+\frac{d}{2}-2\alpha}\big)}$. To this end, we let \begin{equation*}
    \begin{aligned}
        A_{1}&:=-\int_{\mathbb{R}^{d}} \Delta_{j}(u^{(n)}\cdot \nabla u^{(n+1)})\cdot\Delta_{j}u^{(n+1)} dx,\\
A_{2}&:=\int_{\mathbb{R}^{d}}\Delta_{j}(B^{(n)}\cdot \nabla B^{(n)})\cdot \Delta_{j}u^{(n+1)} dx, \\
A_{3}&:=\int_{\mathbb{R}^{d}}\Delta_{j}(\theta^{(n)}e_{d})\cdot \Delta_{j}u^{(n+1)}dx.
    \end{aligned}
\end{equation*}
\begin{lem}\label{u(n+1)}
The following estimate holds 
\begin{equation*}
    \begin{aligned}
     \lVert u^{(n+1)} \rVert_{\tilde{L}^{\infty}\big((0,T);B_{2,1}^{1+\frac{d}{2}-2\alpha}\big)} &\leq \lVert u_{0}^{(n+1)} \rVert_{B_{2,1}^{1+\frac{d}{2}-2\alpha}} + C \delta \lVert u^{(n+1)} \rVert_{\tilde{L}^{\infty}\big((0,T);B_{2,1}^{1+\frac{d}{2}-2\alpha}\big)}\\ & +
    C\lVert B^{(n)} \rVert_{\tilde{L}^{\infty}\big((0,T);B_{2,1}^{1+\frac{d}{2}-2\alpha}\big)} \lVert B^{(n)} \rVert_{L^{1}\big((0,T);B_{2,1}^{1+\frac{d}{2}}\big)} \\ &
      + CT\lVert \theta^{(n)} \rVert_{\tilde{L}^{\infty}\big((0,T);B_{2,1}^{1+\frac{d}{2}-\alpha}\big)}.
\end{aligned}
\end{equation*}
\end{lem}
\begin{proof}
Let $j \geq 0$ be an integer. Applying $\Delta_{j}$ to (\ref{sys2})$_{1}$ and then dotting the equation with
$\Delta_{j}u^{(n+1)}$ we get
\begin{equation}\label{eq1}
    \dfrac{1}{2}\dfrac{d}{dt}  \lVert \Delta_{j}u^{(n+1)} \rVert_{L_{2}}^{2}+\mu \lVert \Lambda^{\alpha}\Delta_{j}u^{(n+1)}\rVert_{L^{2}}^{2}=A_{1}+A_{2}+A_{3}.
\end{equation}
Applying the Bernstein’s inequality, the dissipative part of $(\ref{eq1})$ admit a lower bound where $C_{0} > 0$ is a constant. 
$$
\mu \lVert \Lambda^{\alpha}\Delta_{j}u^{(n+1)}\rVert_{L^{2}}^{2} \geq C_{0}2^{2\alpha j}\lVert \Delta_{j}u^{(n+1)}\rVert_{L^{2}}^{2}.
$$
According to Lemma \ref{lem2}, $A_{1}$ can be bounded by
\begin{equation}\label{eq2}
    \begin{aligned}
           |A_{1}| 
           &\leq  C \lVert \Delta_{j}u^{(n+1)}\rVert_{L^{2}} \bigg(\sum_{m\leq j-1}2^{(1+\frac{d}{2})m} \lVert \Delta_{m} u^{(n)} \rVert _{L^{2}}\sum_{|j-k|\leq 2}\lVert \Delta_{k} u^{(n+1)} \rVert _{L^{2}}\\
           &+\sum_{|j-k|\leq 2}\lVert \Delta_{k} u^{(n)} \rVert _{L^{2}}  \sum_{m\leq j-1}2^{(1+\frac{d}{2})m} \lVert \Delta_{m} u^{(n+1)} \rVert_{L^{2}}\\
  &  +\sum_{k\geq j-4} 2^{j}2^{\frac{d}{2}k}
\lVert \Delta_{k} u^{(n)} \rVert _{L^{2}}\lVert \tilde{\Delta}_{k} u^{(n+1)} \rVert _{L^{2}} \bigg).
    \end{aligned}
\end{equation}
Also, by Lemma \ref{lem2}, $A_{2}$ is bounded by
\begin{equation}\label{eq3}
\begin{aligned}
   |A_{2}|
   &\leq  C \lVert \Delta_{j}u^{(n+1)}\rVert_{L^{2}}\bigg(2^{j} \sum_{m\leq j-1} 2^{\frac{d}{2}m}\lVert \Delta_{m} B^{(n)} \rVert_{L^{2}}\sum_{|j-k|\leq 2} \lVert \Delta_{k} B^{(n)} \rVert_{L^{2}}\\
    &+\sum_{|j-k|\leq 2} \lVert \Delta_{k} B^{(n)} \rVert_{L^{2}} \sum_{m\leq j-1} 2^{(1+\frac{d}{2})m}\lVert \Delta_{m} B^{(n)} \rVert_{L^{2}}\\
&+\sum_{k\geq j-4} 2^{j}2^{\frac{d}{2}k}
\lVert \Delta_{k} B^{(n)} \rVert _{L^{2}}\lVert \tilde{\Delta}_{k} B^{(n)} \rVert _{L^{2}} \bigg).
\end{aligned}
\end{equation}
Applying the H\"older and Bernstein’s inequalities,  $A_{3}$ can be bounded by
\begin{equation}\label{eq4}
    \begin{aligned}
|A_{3}| &\leq C \lVert \Delta_{j} \theta^{(n)} \rVert_{L^{2}}\lVert \Delta_{j}u^{(n+1)}\rVert_{L^{2}} 
    \end{aligned}
\end{equation}
Inserting the estimates (\ref{eq2})-(\ref{eq4} ) into (\ref{eq1}) and eliminating 	$\lVert \Delta_{j}u^{(n+1)} \rVert_{L_{2}}$
	from both side of
the inequality, we get
\begin{equation}\label{eq5}
     \dfrac{d}{dt}\lVert \Delta_{j}u^{(n+1)} \rVert_{L_{2}}+C_{0}2^{2\alpha j}\lVert \Delta_{j}u^{(n+1)}\rVert_{L^{2}}\leq I_{1}+\cdots+I_{7}
\end{equation}
where
\begin{equation*}
    \begin{aligned}
        I_{1}&=C \sum_{m\leq j-1}2^{(1+\frac{d}{2})m} \lVert \Delta_{m} u^{(n)} \rVert _{L^{2}}\sum_{|j-k|\leq 2}\lVert \Delta_{k} u^{(n+1)} \rVert _{L^{2}} 
\\
I_{2}&=C \sum_{|j-k|\leq 2}\lVert \Delta_{k} u^{(n)} \rVert _{L^{2}}  \sum_{m\leq j-1}2^{(1+\frac{d}{2})m} \lVert \Delta_{m} u^{(n+1)} \rVert_{L^{2}}
\\
I_{3}&= C  \sum_{k\geq j-4} 2^{j}2^{\frac{d}{2}k}
\lVert \Delta_{k} u^{(n)} \rVert _{L^{2}}\lVert \tilde{\Delta}_{k} u^{(n+1)} \rVert _{L^{2}}
\\
I_{4}&= C\quad 2^{j} \sum_{m\leq j-1} 2^{\frac{d}{2}m}\lVert \Delta_{m} B^{(n)} \rVert_{L^{2}}\sum_{|j-k|\leq 2} \lVert \Delta_{k} B^{(n)} \rVert_{L^{2}}
\\
I_{5}&= C \sum_{|j-k|\leq 2} \lVert \Delta_{k} B^{(n)} \rVert_{L^{2}} \sum_{m\leq j-1} 2^{(1+\frac{d}{2})m}\lVert \Delta_{m} B^{(n)} \rVert_{L^{2}}
\\
I_{6}&= C \sum_{k\geq j-4} 2^{j}2^{\frac{d}{2}k}
\lVert \Delta_{k} B^{(n)} \rVert _{L^{2}}\lVert \tilde{\Delta}_{k} B^{(n)} \rVert_{L^{2}}
\\
I_{7}&= C \lVert \Delta_{j} \theta^{(n)} \rVert_{L^{2}}
    \end{aligned}
\end{equation*}
Integrating $(\ref{eq5})$ in time for $t \in  [0,T] $ yields
\begin{equation}\label{eq6}
 \lVert \Delta_{j}u^{(n+1)} \rVert_{L_{2}} \leq e^{-C_{0}2^{2\alpha j}t} \lVert \Delta_{j}u_{0}^{(n+1)} \rVert_{L_{2}} + \int_{0}^{t}e^{-C_{0}2^{2\alpha j}(t-\tau)}(I_{1}+\cdots+I_{7}) d\tau 
\end{equation}
For $j=-1$
\begin{equation}\label{eq7}
     \lVert \Delta_{-1}u^{(n+1)} \rVert_{L_{2}} \leq \lVert \Delta_{-1}u_{0}^{(n+1)} \rVert_{L_{2}} + \int_{0}^{t}(I_{1}+\cdots+I_{7}) d\tau 
\end{equation}
Taking the $L^{\infty}(0,T)$ of $(\ref{eq6})$ and $(\ref{eq7})$, then multiplying by $ 2^{(1+\frac{d}{2}-2\alpha)j}$
and summing up the resulting inequalities with respect to $j$, we have
\begin{equation}\label{eq8}
     \lVert u^{(n+1)} \rVert_{\tilde{L}^{\infty}\big((0,T);B_{2,1}^{1+\frac{d}{2}-2\alpha}\big)} \leq \lVert u_{0}^{(n+1)} \rVert_{B_{2,1}^{1+\frac{d}{2}-2\alpha}} + \sum_{j \geq -1}  2^{(1+\frac{d}{2}-2\alpha)j }\int_{0}^{t}(I_{1}+\cdots+I_{7}) d\tau 
\end{equation}
where we have used the fact $e^{-C_{0}2^{2\alpha j(t-\tau)}}\leq 1$.
Now, we estimate the terms involving $I_{1}$ through $I_{7}$. By H\"older inequality, $I_{1}$ can be estimated as follows
\begin{equation}\label{eq9}
    \begin{aligned}
         \sum_{j \geq -1}  2^{(1+\frac{d}{2}-2\alpha)j }\int_{0}^{t}I_{1}d\tau &\leq C \int_{0}^{T}\sum_{j \geq -1}  2^{(1+\frac{d}{2}-2\alpha)j} \lVert \Delta_{j} u^{(n+1)}\rVert_{L^{2}}\sum_{m\leq j-1}2^{(1+\frac{d}{2})m} \lVert \Delta_{m} u^{(n)} \rVert _{L^{2}} d\tau
    \\
  & \leq C \lVert u^{(n+1)} \rVert_{\tilde{L}^{\infty}\big((0,T);B_{2,1}^{1+\frac{d}{2}-2\alpha}\big)} \lVert u^{(n)} \rVert_{L^{1}\big((0,T);B_{2,1}^{1+\frac{d}{2}}\big)}.
    \end{aligned}
\end{equation}   
Using Young's inequality for series convolution, the term with $I_{2}$ can be estimated as follows:
\begin{equation}\label{eq10}
    \begin{aligned}
           \sum_{j \geq -1}  2^{(1+\frac{d}{2}-2\alpha)j }\int_{0}^{t}I_{2}d\tau &\leq C \int_{0}^{T} \sum_{j\geq 1} 2^{(1+\frac{d}{2})j} \lVert \Delta_{j}u^{(n)}\rVert_{L^{2}} \sum_{m \leq j-1} 2^{-2\alpha(j-m)}2^{(1+\frac{d}{2}-2\alpha)m} \lVert \Delta_{m} u^{(n+1)}\rVert_{L^{2}} d\tau
     \\
   &\leq C \lVert u^{(n+1)} \rVert_{\tilde{L}^{\infty}\big((0,T);B_{2,1}^{1+\frac{d}{2}-2\alpha}\big)} \lVert u^{(n)} \rVert_{L^{1}\big((0,T);B_{2,1}^{1+\frac{d}{2}}\big)}.
    \end{aligned}
\end{equation}
The term $I_{3}$ can be bounded as $I_{1}- I_{2}$
\begin{equation}\label{eq11}
    \begin{aligned}
         \sum_{j \geq -1}  2^{(1+\frac{d}{2}-2\alpha)j }\int_{0}^{t}I_{3} d\tau &\leq C \int_{0}^{T} \sum_{j \geq -1}  2^{(1+\frac{d}{2}-2\alpha)j }\sum_{k\geq j-4} 2^{j}2^{\frac{d}{2}k}
\lVert \Delta_{k} u^{(n)} \rVert _{L^{2}}\lVert \tilde{\Delta}_{k} u^{(n+1)} \rVert _{L^{2}} d \tau
\\
&\leq C \int_{0}^{T}  \sum_{j \geq -1} \sum_{k\geq j-4} 2^{(2+\frac{d}{2}-2\alpha)(j-k)}2^{(1+\frac{d}{2})k}\lVert \Delta_{k} u^{(n)} \rVert _{L^{2}} 2^{(1+\frac{d}{2}-2\alpha)k}\lVert \tilde{\Delta}_{k} u^{(n+1)} \rVert _{L^{2}} d \tau
    \\
   &\leq C \lVert u^{(n+1)} \rVert_{\tilde{L}^{\infty}\big((0,T);B_{2,1}^{1+\frac{d}{2}-2\alpha}\big)} \lVert u^{(n)} \rVert_{L^{1}\big((0,T);B_{2,1}^{1+\frac{d}{2}}\big)}.
   \end{aligned}
\end{equation}
where we have used Young’s inequality for series convolution and we need $\alpha < 1 + \frac{d}{4}$
. The term with $I_{4}$ is bounded by
\begin{equation}\label{eq12}
    \begin{aligned}
           \sum_{j \geq -1}  2^{(1+\frac{d}{2}-2\alpha)j }\int_{0}^{t}I_{4} d\tau &\leq C \int_{0}^{T} \sum_{j \geq -1}  2^{(1+\frac{d}{2}-2\alpha)j }  2^{j} \lVert \Delta_{j} B^{(n)} \rVert_{L^{2}}\sum_{m\leq j-1} 2^{\frac{d}{2}m}\lVert \Delta_{m} B^{(n)} \rVert_{L^{2}} 
    \\
   & \leq  C \int_{0}^{T} \sum_{j \geq -1}  2^{(1+\frac{d}{2})j }\lVert \Delta_{j} B^{(n)} \rVert_{L^{2}}\sum_{m \leq j-1} 2^{(2 \alpha-1)(m-j)} 2^{(1+\frac{d}{2}-2\alpha)m }\lVert \Delta_{m} B^{(n)} \rVert_{L^{2}} 
    \\
    &\leq C \int_{0}^{T} \lVert B^{(n)}\rVert_{B_{2,1}^{1+\frac{d}{2}}}\lVert B^{(n)}\rVert_{B_{2,1}^{1+\frac{d}{2}-2\alpha}} d\tau
    \\
    & \leq C \lVert B^{(n)} \rVert_{\tilde{L}^{\infty}\big((0,T);B_{2,1}^{1+\frac{d}{2}-2\alpha}\big)} \lVert B^{(n)} \rVert_{L^{1}\big((0,T);B_{2,1}^{1+\frac{d}{2}}\big)}.
    \end{aligned}
\end{equation}
 The term with $I_{5}$ is estimated as follows
\begin{equation}\label{eq13}
    \begin{aligned}
         \sum_{j \geq -1}  2^{(1+\frac{d}{2}-2\alpha)j }\int_{0}^{t} I_{5} 
 d\tau &\leq  C \int_{0}^{T} \sum_{j\geq 1} 2^{(1+\frac{d}{2})j} \lVert \Delta_{j}B^{(n)}\rVert_{L^{2}} \sum_{m \leq j-1} 2^{-2\alpha(j-m)}2^{(1+\frac{d}{2}-2\alpha)m} \lVert \Delta_{m} B^{(n)}\rVert_{L^{2}} d\tau
     \\
   &\leq C \lVert B^{(n)} \rVert_{\tilde{L}^{\infty}\big((0,T);B_{2,1}^{1+\frac{d}{2}-2\alpha}\big)} \lVert B^{(n)} \rVert_{L^{1}\big((0,T);B_{2,1}^{1+\frac{d}{2}}\big)}. 
   \end{aligned}
\end{equation}
The term with $I_{6}$ is estimated as follows:
\begin{equation}\label{eq14}
    \begin{aligned}
            \sum_{j \geq -1}  2^{(1+\frac{d}{2}-2\alpha)j }\int_{0}^{t} I_{6} 
 d\tau &\leq  C \int_{0}^{T} \sum_{j\geq 1} 2^{(1+\frac{d}{2}-2\alpha)j}\sum_{k\geq j-4} 2^{j}2^{\frac{d}{2}k}
\lVert \Delta_{k} B^{(n)} \rVert _{L^{2}}\lVert \tilde{\Delta}_{k} B^{(n)} \rVert _{L^{2}}
\\
&\leq C \int_{0}^{T}  \sum_{j \geq -1} \sum_{k\geq j-4} 2^{(2+\frac{d}{2}-2\alpha)(j-k)}2^{(1+\frac{d}{2})k}\lVert \Delta_{k} B^{(n)} \rVert _{L^{2}} 2^{(1+\frac{d}{2}-2\alpha)k}\lVert \tilde{\Delta}_{k} B^{(n)} \rVert _{L^{2}} d \tau
    \\
   &\leq C \lVert B^{(n)} \rVert_{\tilde{L}^{\infty}\big((0,T);B_{2,1}^{1+\frac{d}{2}-2\alpha}\big)} \lVert B^{(n)} \rVert_{L^{1}\big((0,T);B_{2,1}^{1+\frac{d}{2}}\big)}.
    \end{aligned}
\end{equation}
The term $I_{7}$ is estimated as follows:
\begin{equation}\label{eq15}
    \begin{aligned}
        \sum_{j \geq -1}  2^{(1+\frac{d}{2}-2\alpha)j }\int_{0}^{t} I_{7} 
 d\tau &\leq  C \int_{0}^{T} \sum_{j\geq 1}2^{-\alpha j} 2^{(1+\frac{d}{2}-\alpha)j}\lVert \Delta_{j} \theta^{(n)} \rVert_{L^{2}} d \tau \\
 &\leq CT \lVert \theta^{(n)}\rVert_{\tilde{L}^{\infty}\big((0,T);B_{2,1}^{1+\frac{d}{2}-\alpha}\big)}.\\
    \end{aligned}
\end{equation}    
To obtain the intended outcome, we gather the estimations (\ref{eq9})-(\ref{eq15}) and insert them into (\ref{eq8}).
\end{proof}
\par New, we estimate the follwing term $\lVert B^{(n+1)}\rVert_{\tilde{L}^{\infty}((0,T);B_{2,1}^{1+\frac{d}{2}-2\alpha}(\mathbb{R}^{d}))}$. To this end, let 
\begin{equation*}
    \begin{aligned}
         B_{1}&:=-\int_{\mathbb{R}^{d}} \Delta_{j}(u^{(n)}\cdot \nabla B^{(n+1)})\cdot\Delta_{j}B^{(n+1)} dx,
         \\
B_{2}&:=\int_{\mathbb{R}^{d}}\Delta_{j}(B^{(n)}\cdot \nabla u^{(n)})\cdot \Delta_{j}B^{(n+1)} dx,
    \end{aligned}
\end{equation*}
and
\begin{equation*}
    \begin{aligned}
        J_{1}&:=C \sum_{m\leq j-1}2^{(1+\frac{d}{2})m} \lVert \Delta_{m} B^{(n)} \rVert _{L^{2}}\sum_{|j-k|\leq 2}\lVert \Delta_{k} B^{(n+1)} \rVert _{L^{2}}, 
\\
J_{2}&:=C \sum_{|j-k|\leq 2}\lVert \Delta_{k} B^{(n)} \rVert _{L^{2}}  \sum_{m\leq j-1}2^{(1+\frac{d}{2})m} \lVert \Delta_{m} B^{(n+1)} \rVert_{L^{2}},
\\
J_{3}&:= C  \sum_{k\geq j-4} 2^{j}2^{\frac{d}{2}k}
\lVert \Delta_{k} B^{(n)} \rVert _{L^{2}}\lVert \tilde{\Delta}_{k} B^{(n+1)} \rVert _{L^{2}},
\\
J_{4}&:= C~~2^{j} \sum_{m\leq j-1} 2^{\frac{d}{2}m}\lVert \Delta_{m} B^{(n)} \rVert_{L^{2}}\sum_{|j-k|\leq 2} \lVert \Delta_{k} u^{(n)} \rVert_{L^{2}},\\
J_{5}&:= C \sum_{|j-k|\leq 2} \lVert \Delta_{k} B^{(n)} \rVert_{L^{2}} \sum_{m\leq j-1} 2^{(1+\frac{d}{2})m}\lVert \Delta_{m} u^{(n)} \rVert_{L^{2}},\\
J_{6}&:= C \sum_{k\geq j-4} 2^{j}2^{\frac{d}{2}k}\lVert \Delta_{k} B^{(n)} \rVert _{L^{2}}\lVert \tilde{\Delta}_{k} u^{(n)} \rVert _{L^{2}}.
    \end{aligned}
\end{equation*}
\begin{lem}
The following estimate holds:
\begin{equation}\label{eq29}
 \begin{aligned}
 \lVert B^{(n+1)} \rVert_{\tilde{L}^{\infty}\big((0,T);B_{2,1}^{1+\frac{d}{2}-2\alpha}\big)} &\leq \lVert B_{0}^{(n+1)} \rVert_{B_{2,1}^{1+\frac{d}{2}-2\alpha}} +  C \lVert B^{(n)} \rVert_{L^{1}\big((0,T);B_{2,1}^{1+\frac{d}{2}}\big)} \\
&\times\lVert B^{(n+1)} \rVert_{\tilde{L}^{\infty}\big((0,T);B_{2,1}^{1+\frac{d}{2}-2\alpha}\big)}\\
&+ C\lVert B^{(n)} \rVert_{\tilde{L}^{\infty}\big((0,T);B_{2,1}^{1+\frac{d}{2}-2\alpha}\big)} \lVert u^{(n)} \rVert_{L^{1}\big((0,T);B_{2,1}^{1+\frac{d}{2}}\big)}\\
&+  C\lVert u^{(n)} \rVert_{\tilde{L}^{\infty}\big((0,T);B_{2,1}^{1+\frac{d}{2}-2\alpha}\big)} \lVert B^{(n)} \rVert_{L^{1}\big((0,T);B_{2,1}^{1+\frac{d}{2}}\big)}.
 \end{aligned}
\end{equation}    
\end{lem}
\begin{proof}
Let $j \geq 0$ be an integer. Applying $\Delta_{j}$ to (\ref{sys2})$_{2}$ and then dotting the equation with $\Delta_{j}B^{(n+1)}$ we get
\begin{equation}\label{eq17}
     \dfrac{1}{2}\dfrac{d}{dt}  \lVert \Delta_{j}B^{(n+1)} \rVert_{L_{2}}^{2}+\eta \lVert \Lambda^{\alpha}\Delta_{j}B^{(n+1)}\rVert_{L^{2}}^{2}=B_{1}+B_{2},
\end{equation}
Applying the Bernstein’s inequality, the dissipative part of $(\ref{eq17})$ admit a lower bound as follows: 
\begin{equation*}\label{eq19}
    \eta \lVert \Lambda^{\alpha}\Delta_{j}B^{(n+1)}\rVert_{L^{2}}^{2} \geq C_{1}2^{2\alpha j}\lVert \Delta_{j}B^{(n+1)}\rVert_{L^{2}}^{2},\;\;\mbox{for some positive constant}\;\;C_{1}.
\end{equation*}
Drawing on an argument like we did to derive (\ref{eq5}), we have
\begin{equation}\label{eq20}
     \dfrac{d}{dt}\lVert \Delta_{j}B^{(n+1)} \rVert_{L_{2}}+C_{1}2^{2\alpha j}\lVert \Delta_{j}B^{(n+1)}\rVert_{L^{2}}\leq J_{1}+\cdots+J_{6}.
\end{equation}
Lemma \ref{lem2} states that $B_{1}$ is bounded by
\begin{eqnarray*}
\begin{aligned}
    |B_{1}| &\leq  C \lVert \Delta_{j}B^{(n+1)}\rVert_{L^{2}} \bigg(\sum_{m\leq j-1}2^{(1+\frac{d}{2})m} \lVert \Delta_{m} B^{(n)} \rVert _{L^{2}}\sum_{|j-k|\leq 2}\lVert \Delta_{k} B^{(n+1)} \rVert _{L^{2}}\\
    &+\sum_{|j-k|\leq 2}\lVert \Delta_{k} B^{(n)} \rVert _{L^{2}}  \sum_{m\leq j-1}2^{(1+\frac{d}{2})m} \lVert \Delta_{m} B^{(n+1)} \rVert_{L^{2}}\\
   & +\sum_{k\geq j-4} 2^{j}2^{\frac{d}{2}k}
\lVert \Delta_{k} B^{(n)} \rVert _{L^{2}}\lVert \tilde{\Delta}_{k} B^{(n+1)} \rVert _{L^{2}} \bigg).
\end{aligned}  
\end{eqnarray*}
Also, by Lemma \ref{lem2}, $B_{2}$ is bounded by
\begin{eqnarray*}
\begin{aligned}
   |B_{2}| &\leq  C \lVert \Delta_{j}B^{(n+1)}\rVert_{L^{2}}\bigg(2^{j} \sum_{m\leq j-1} 2^{\frac{d}{2}m}\lVert \Delta_{m} B^{(n)} \rVert_{L^{2}}\sum_{|j-k|\leq 2} \lVert \Delta_{k} u^{(n)} \rVert_{L^{2}}\\
   &+\sum_{|j-k|\leq 2} \lVert \Delta_{k} B^{(n)} \rVert_{L^{2}} \sum_{m\leq j-1} 2^{(1+\frac{d}{2})m}\lVert \Delta_{m} u^{(n)} \rVert_{L^{2}}\\
&+\sum_{k\geq j-4} 2^{j}2^{\frac{d}{2}k}
\lVert \Delta_{k} B^{(n)} \rVert _{L^{2}}\lVert \tilde{\Delta}_{k} u^{(n)} \rVert _{L^{2}} \bigg).
\end{aligned}
\end{eqnarray*}
Integrating (\ref{eq20}) in time for $ t \in [0,T] $  yields
\begin{equation}\label{eq21}
    \lVert \Delta_{j}B^{(n+1)} \rVert_{L_{2}} \leq e^{-C_{1}2^{2\alpha j}t} \lVert \Delta_{j}B_{0}^{(n+1)} \rVert_{L_{2}} + \int_{0}^{t}e^{-C_{1}2^{2\alpha j}(t-\tau)}(J_{1}+\cdots+J_{6}) d\tau.
\end{equation}
For $j=-1$, arguing similarly as (\ref{eq9})-(\ref{eq15}), and taking the $L^{\infty}(0,T)$ of (\ref{eq21}), multiplying
by $2^{1+\frac{d}{2}-2\alpha}$ and summing over $j$, we deduce
\begin{equation}\label{eq22}
        \lVert B^{(n+1)} \rVert_{\tilde{L}^{\infty}\big((0,T);B_{2,1}^{1+\frac{d}{2}-2\alpha}\big)} \leq \lVert B_{0}^{(n+1)} \rVert_{B_{2,1}^{1+\frac{d}{2}-2\alpha}} + \sum_{j \geq -1}  2^{(1+\frac{d}{2}-2\alpha)j }\int_{0}^{t}(J_{1}+\cdots+J_{6}) d\tau.
\end{equation}
The terms involving $J_{1}$ through $J_{6}$ can be bounded as $I_{1} - I_{7}$
\begin{equation}\label{eq23}
    \begin{aligned}
          \sum_{j \geq -1}  2^{(1+\frac{d}{2}-2\alpha)j }\int_{0}^{t}J_{1}d\tau &\leq C \int_{0}^{T}\sum_{j \geq -1}  2^{(1+\frac{d}{2}-2\alpha)j} \lVert \Delta_{j} B^{(n+1)}\rVert_{L^{2}}\sum_{m\leq j-1}2^{(1+\frac{d}{2})m} \lVert \Delta_{m} B^{(n)} \rVert _{L^{2}} d\tau
    \\
  & \leq C \lVert B^{(n+1)} \rVert_{\tilde{L}^{\infty}\big((0,T);B_{2,1}^{1+\frac{d}{2}-2\alpha}\big)} \lVert B^{(n)} \rVert_{L^{1}\big((0,T);B_{2,1}^{1+\frac{d}{2}}\big)} .
    \end{aligned}
\end{equation}
\begin{equation}\label{eq24}
    \begin{aligned}
        \sum_{j \geq -1}  2^{(1+\frac{d}{2}-2\alpha)j }\int_{0}^{t}J_{2}d\tau &\leq C \int_{0}^{T} \sum_{j\geq 1} 2^{(1+\frac{d}{2})j} \lVert \Delta_{j}B^{(n)}\rVert_{L^{2}} \sum_{m \leq j-1} 2^{-2\alpha(j-m)}2^{(1+\frac{d}{2}-2\alpha)m} \lVert \Delta_{m} B^{(n+1)}\rVert_{L^{2}} d\tau
     \\
  & \leq C \lVert B^{(n+1)} \rVert_{\tilde{L}^{\infty}\big((0,T);B_{2,1}^{1+\frac{d}{2}-2\alpha}\big)} \lVert B^{(n)} \rVert_{L^{1}\big((0,T);B_{2,1}^{1+\frac{d}{2}}\big)}.
   \end{aligned}
\end{equation}
So, the term $J_{3}$ can be bounded as $J_{1}- J_{2}$
\begin{equation}\label{eq25}
    \begin{aligned}
         \sum_{j \geq -1}  2^{(1+\frac{d}{2}-2\alpha)j }\int_{0}^{t}J_{3} d\tau &\leq C \int_{0}^{T} \sum_{j \geq -1}  2^{(1+\frac{d}{2}-2\alpha)j }\sum_{k\geq j-4} 2^{j}2^{\frac{d}{2}k}
\lVert \Delta_{k} B^{(n)} \rVert _{L^{2}}\lVert \tilde{\Delta}_{k} B^{(n+1)} \rVert _{L^{2}} d \tau
\\
&\leq C \int_{0}^{T}  \sum_{j \geq -1} \sum_{k\geq j-4} 2^{(2+\frac{d}{2}-2\alpha)(j-k)}2^{(1+\frac{d}{2})k}\lVert \Delta_{k} B^{(n)} \rVert _{L^{2}} 2^{(1+\frac{d}{2}-2\alpha)k}\lVert \tilde{\Delta}_{k} B^{(n+1)} \rVert _{L^{2}} d \tau
    \\
   &\leq C \lVert B^{(n+1)} \rVert_{\tilde{L}^{\infty}\big((0,T);B_{2,1}^{1+\frac{d}{2}-2\alpha}\big)} \lVert B^{(n)} \rVert_{L^{1}\big((0,T);B_{2,1}^{1+\frac{d}{2}}\big)}. 
   \end{aligned}
\end{equation}
Moreover, the term with $J_{4}$ is bounded by
 \begin{equation}\label{eq26}
     \begin{aligned}
         \sum_{j \geq -1}  2^{(1+\frac{d}{2}-2\alpha)j }\int_{0}^{t}J_{4} d\tau &\leq C \int_{0}^{T} \sum_{j \geq -1}  2^{(1+\frac{d}{2}-2\alpha)j }  2^{j} \lVert \Delta_{j} u^{(n)} \rVert_{L^{2}}\sum_{m\leq j-1} 2^{\frac{d}{2}m}\lVert \Delta_{m} B^{(n)} \rVert_{L^{2}} 
    \\
   & \leq  C \int_{0}^{T} \sum_{j \geq -1}  2^{(1+\frac{d}{2})j }\lVert \Delta_{j} u^{(n)} \rVert_{L^{2}}\sum_{m \leq j-1} 2^{(2 \alpha-1)(m-j)} 2^{(1+\frac{d}{2}-2\alpha)m }\lVert \Delta_{m} B^{(n)} \rVert_{L^{2}} 
    \\
    &\leq C \int_{0}^{T} \lVert u^{(n)}\rVert_{B_{2,1}^{1+\frac{d}{2}}}\lVert B^{(n)}\rVert_{B_{2,1}^{1+\frac{d}{2}-2\alpha}} d\tau
    \\
     &\leq C \lVert B^{(n)} \rVert_{\tilde{L}^{\infty}\big((0,T);B_{2,1}^{1+\frac{d}{2}-2\alpha}\big)} \lVert u^{(n)} \rVert_{L^{1}\big((0,T);B_{2,1}^{1+\frac{d}{2}}\big)}.
     \end{aligned}
 \end{equation}
Also, the term with $J_{5}$ is estimated as follows
\begin{equation}\label{eq27}
    \begin{aligned}
           \sum_{j \geq -1}  2^{(1+\frac{d}{2}-2\alpha)j }\int_{0}^{t}J_{5} 
 d\tau &\leq  C \int_{0}^{T} \sum_{j\geq 1} 2^{(1+\frac{d}{2})j} \lVert \Delta_{j}B^{(n)}\rVert_{L^{2}} \sum_{m \leq j-1} 2^{-2\alpha(j-m)}2^{(1+\frac{d}{2}-2\alpha)m} \lVert \Delta_{m} u^{(n)}\rVert_{L^{2}} d\tau
     \\
  & \leq C \lVert u^{(n)} \rVert_{\tilde{L}^{\infty}\big((0,T);B_{2,1}^{1+\frac{d}{2}-2\alpha}\big)} \lVert B^{(n)} \rVert_{L^{1}\big((0,T);B_{2,1}^{1+\frac{d}{2}}\big)}.
\end{aligned}
\end{equation}
 Finally, the term with $J_{6}$ is estimated as follows
\begin{equation}\label{eq28}
    \begin{aligned}
           \sum_{j \geq -1}  2^{(1+\frac{d}{2}-2\alpha)j }\int_{0}^{t} J_{6} 
 d\tau &\leq  C \int_{0}^{T} \sum_{j\geq 1} 2^{(1+\frac{d}{2}-2\alpha)j}\sum_{k\geq j-4} 2^{j}2^{\frac{d}{2}k}
\lVert \Delta_{k} B^{(n)} \rVert _{L^{2}}\lVert \tilde{\Delta}_{k} u^{(n)} \rVert _{L^{2}}
\\
&\leq C \int_{0}^{T}  \sum_{j \geq -1} \sum_{k\geq j-4} 2^{(2+\frac{d}{2}-2\alpha)(j-k)}2^{(1+\frac{d}{2})k}\lVert \Delta_{k} B^{(n)} \rVert _{L^{2}} 2^{(1+\frac{d}{2}-2\alpha)k}\lVert \tilde{\Delta}_{k} u^{(n)} \rVert _{L^{2}} d \tau
    \\
  &\leq C \lVert B^{(n)} \rVert_{\tilde{L}^{\infty}\big((0,T);B_{2,1}^{1+\frac{d}{2}-2\alpha}\big)} \lVert u^{(n)} \rVert_{L^{1}\big((0,T);B_{2,1}^{1+\frac{d}{2}}\big)}.
\end{aligned}
\end{equation}
\end{proof}
We will estimate the term $\lVert \theta^{(n+1)}\rVert_{\tilde{L}^{\infty}((0,T);B_{2,1}^{1+\frac{d}{2}-\alpha}(\mathbb{R}^{d}))}$. To this end, let 
\begin{equation*}
    \begin{aligned}
       C_{1}&=-\int_{\mathbb{R}^{d}} \Delta_{j}(u^{(n)}\cdot \nabla \theta^{(n+1)})\cdot\Delta_{j}\theta^{(n+1)} dx,\\
C_{2}&=\int_{\mathbb{R}^{d}}\Delta_{j}(u^{(n)}\cdot e_{d})\cdot \Delta_{j}\theta^{(n+1)} dx, 
    \end{aligned}
\end{equation*}
and
\begin{equation*}
    \begin{aligned}
        K_{1}&=C \sum_{m\leq j-1}2^{(1+\frac{d}{2})m} \lVert \Delta_{m} u^{(n)} \rVert _{L^{2}}\sum_{|j-k|\leq 2}\lVert \Delta_{k} \theta^{(n+1)} \rVert _{L^{2}}, 
\\
K_{2}&=C \sum_{|j-k|\leq 2}\lVert \Delta_{k} u^{(n)} \rVert _{L^{2}}  \sum_{m\leq j-1}2^{(1+\frac{d}{2})m} \lVert \Delta_{m} \theta^{(n+1)} \rVert_{L^{2}},
\\
K_{3}&= C  \sum_{k\geq j-4} 2^{j}2^{\frac{d}{2}k}
\lVert \Delta_{k} u^{(n)} \rVert _{L^{2}}\lVert \tilde{\Delta}_{k} \theta^{(n+1)} \rVert _{L^{2}},
\\
K_{4}&= C \lVert \Delta_{j} u^{(n)} \rVert_{L^{2}}. 
    \end{aligned}
\end{equation*}
\begin{lem}\label{estimatetheta(n+1)}
The following estimate holds
Collecting the estimates (\ref{eq38})-(\ref{eq40}) and inserting them into (\ref{eq37}), we arrive at
\begin{equation*}
    \begin{aligned}
          \lVert \theta^{(n+1)} \rVert_{\tilde{L}^{\infty}\big((0,T);B_{2,1}^{1+\frac{d}{2}-\alpha}\big)} &\leq \lVert \theta_{0}^{(n+1)} \rVert_{B_{2,1}^{1+\frac{d}{2}-\alpha}}+C\lVert u^{(n)} \rVert_{L^{1}\big((0,T);B_{2,1}^{1+\frac{d}{2}}\big)}  \lVert \theta^{(n+1)} \rVert_{\tilde{L}^{\infty}\big((0,T);B_{2,1}^{1+\frac{d}{2}-\alpha}\big)}\\ 
					&+ C  \lVert u^{(n)} \rVert_{L^{1}\big((0,T);B_{2,1}^{1+\frac{d}{2}}\big)} \\.
    \end{aligned}
\end{equation*}

\end{lem}
\begin{proof}
Let $j \geq -1  $ be an integer. Applying $\Delta_{j}$ to (\ref{sys2})$_{3}$ and then dotting the equation with $\Delta_{j} \theta^{(n+1)}$,
we have
\begin{equation}\label{eq30}
    \dfrac{1}{2}\dfrac{d}{dt} \lVert \Delta_{j}\theta^{(n+1)} \rVert_{L^{2}}^{2} = C_{1}+C_{2}.
\end{equation}
According to Lemma (\ref{lem2}), it holds that
\begin{equation}\label{eq32}
    \begin{aligned}
     |C_{1}| &\leq  C \lVert \Delta_{j}\theta^{(n+1)}\rVert_{L^{2}} \bigg(\sum_{m\leq j-1}2^{(1+\frac{d}{2})m} \lVert \Delta_{m} u^{(n)} \rVert _{L^{2}}\sum_{|j-k|\leq 2}\lVert \Delta_{k}\theta^{(n+1)} \rVert _{L^{2}}\\
    &+\sum_{|j-k|\leq 2}\lVert \Delta_{k} u^{(n)} \rVert _{L^{2}}  \sum_{m\leq j-1}2^{(1+\frac{d}{2})m} \lVert \Delta_{m} \theta^{(n+1)} \rVert_{L^{2}}\\
   & +\sum_{k\geq j-4} 2^{j}2^{\frac{d}{2}k}
\lVert \Delta_{k} u^{(n)} \rVert _{L^{2}}\lVert \tilde{\Delta}_{k} \theta^{(n+1)} \rVert _{L^{2}} \bigg).
\end{aligned}
\end{equation}
By H\"older inequality, $C_{2}$ can be bounded by
\begin{equation}\label{eq33}
    |C_{2}|\leq C \lVert \Delta_{j} u^{(n)}\rVert_{L^{2}}\lVert \Delta_{j} \theta^{(n+1)}\rVert_{L^{2}}
\end{equation}
Inserting the estimates (\ref{eq32}) and (\ref{eq33}) into the equality (\ref{eq30}), then eliminating $\lVert \Delta_{j} \theta^{(n+1)}\rVert_{L^{2}}$ from both sides of the inequality, we have
\begin{equation}\label{eq34}
      \frac{d}{dt} \lVert \Delta_{j} \theta^{(n+1)}\rVert_{L^{2}} \leq K_{1}+\cdots+K_{4}
\end{equation}
Integrating (\ref{eq34}) in time for $ t\in [0,T]$ we have
\begin{equation}\label{eq36}
    \lVert \Delta_{j} \theta^{(n+1)}\rVert_{L^{2}} \leq  \lVert \Delta_{j} \theta_{0}^{(n+1)}\rVert_{L^{2}} + \int_{0}^{t} (K_{1}+\cdots+K_{4} ) d\tau.
\end{equation}
Taking the $L^{\infty}(0,T)$  of (\ref{eq36}), multiplying by $2^{(1+\frac{d}{2}-\alpha)j}$ and summing over $j$ , we get
\begin{equation}\label{eq37}
      \lVert \theta^{(n+1))} \rVert_{\tilde{L}^{\infty}\big((0,T);B_{2,1}^{1+\frac{d}{2}-\alpha}\big)}  \leq \lVert\theta_{0}^{(n+1)}\rVert_{B_{2,1}^{1+\frac{d}{2}-\alpha}} +\sum_{j\geq -1} 2^{(1+\frac{d}{2}-\alpha)j}\int_{0}^{t} (K_{1}+\cdots+K_{4} ) d\tau.
\end{equation}
The terms involving $K_{1}$ through $K_{4} $ can be bounded as $I_{1}-I_{7}$
\begin{equation}\label{eq38}
    \begin{aligned}
          \sum_{j\geq -1} 2^{(1+\frac{d}{2}-\alpha)j} \int_{0}^{T} K_{1} d\tau &\leq C \int_{0}^{T} \sum_{j\geq -1} 2^{(1+\frac{d}{2}-\alpha)j} \lVert \Delta_{j} \theta^{(n+1)} \rVert _{L^{2}} \sum_{m\leq j-1}2^{(1+\frac{d}{2})m} \lVert \Delta_{m} u^{(n)} \rVert _{L^{2}} 
    \\
    &\leq C \lVert \theta^{(n+1)} \rVert_{\tilde{L}^{\infty}\big((0,T);B_{2,1}^{1+\frac{d}{2}-\alpha}\big)} \lVert u^{(n)} \rVert_{L^{1}\big((0,T);B_{2,1}^{1+\frac{d}{2}}\big)}, \\
    \end{aligned}
\end{equation}
\begin{equation}\label{eqd}
    \begin{aligned}
          \sum_{j\geq -1} 2^{(1+\frac{d}{2}-\alpha)j} \int_{0}^{T} K_{2} d\tau &\leq C \int_{0}^{T} \sum_{j\geq -1} 2^{(1+\frac{d}{2}-\alpha)j} \lVert \Delta_{j} u^{(n)} \rVert _{L^{2}} \sum_{m\leq j-1}2^{(1+\frac{d}{2})m} \lVert \Delta_{m} \theta^{(n+1)} \rVert _{L^{2}} 
    \\
    &\leq C \int_{0}^{T} \sum_{j\geq -1} 2^{(1+\frac{d}{2})j} \lVert \Delta_{j} u^{(n)} \rVert _{L^{2}} \sum_{m\leq j-1} 2^{-\alpha(j-m)}2^{(1+\frac{d}{2}-\alpha)m} \lVert \Delta_{m} \theta^{(n+1)} \rVert _{L^{2}}    \\
    &\leq C \lVert \theta^{(n+1)} \rVert_{\tilde{L}^{\infty}\big((0,T);B_{2,1}^{1+\frac{d}{2}-\alpha}\big)} \lVert u^{(n)} \rVert_{L^{1}\big((0,T);B_{2,1}^{1+\frac{d}{2}}\big)}, \\
    \end{aligned}
\end{equation}

  \begin{equation}\label{eq39}
      \begin{aligned}
            \sum_{j\geq -1} 2^{(1+\frac{d}{2}-\alpha)j} \int_{0}^{T} K_{3} d\tau &\leq C \int_{0}^{T}  \sum_{j\geq -1} 2^{(1+\frac{d}{2}-\alpha)j} \sum_{k\geq j-4} 2^{j}2^{\frac{d}{2}k}
\lVert \Delta_{k} u^{(n)} \rVert _{L^{2}}\lVert \tilde{\Delta}_{k} \theta^{(n+1)} \rVert _{L^{2}} d\tau
\\
&\leq
C \int_{0}^{T} \sum_{j\geq -1} \sum_{k\geq j-4} 2^{(2+\frac{d}{2}-\alpha)(j-k)} 2^{(1+\frac{d}{2})k}\lVert \Delta_{k} u^{(n)} \rVert _{L^{2}} 2^{(1+\frac{d}{2}-\alpha)k}\lVert \tilde{\Delta}_{k} \theta^{(n+1)} \rVert _{L^{2}} d\tau
 \\
& \leq C \lVert \theta^{(n+1)} \rVert_{\tilde{L}^{\infty}\big((0,T);B_{2,1}^{1+\frac{d}{2}-\alpha}\big)} \lVert u^{(n)} \rVert_{L^{1}\big((0,T);B_{2,1}^{1+\frac{d}{2}}\big)},
      \end{aligned}
  \end{equation}

\begin{equation}\label{eq40}
    \begin{aligned}
      \sum_{j\geq -1} 2^{(1+\frac{d}{2}-\alpha)j} \int_{0}^{T} K_{4} d\tau 
      &\leq C \int_{0}^{T} \sum_{j\geq -1} 2^{-\alpha j} 2^{(1+\frac{d}{2})j} \lVert \Delta_{j} u^{(n)}\rVert_{L^{2}} d\tau\\ &\leq C  \lVert u^{(n)} \rVert_{L^{1}\big((0,T);B_{2,1}^{1+\frac{d}{2}}\big)} \\
    \end{aligned}
\end{equation}
 \end{proof}
\par Now, we estimate the term $\lVert u^{(n+1)}\rVert_{L^{1}((0,T);B_{2,1}^{1+\frac{d}{2}}(\mathbb{R}^{d}))}$
\begin{lem}\label{u2(n+1)}
The following estimate holds
\begin{equation*}
    \begin{aligned}
\lVert  u^{(n+1)}\rVert_{L^{1}(0,T,B_{2,1}^{1+\frac{d}{2}})} \leq &\dfrac{\delta}{2}
		+ CT\lVert u^{(n+1)} \rVert_{\tilde{L}^{\infty}\big((0,T);B_{2,1}^{1+\frac{d}{2}-2\alpha}\big)}  \lVert u^{(n)} \rVert_{L^{1}\big((0,T);B_{2,1}^{1+\frac{d}{2}}\big)}\\
		&+CT \lVert B^{(n)} \rVert_{\tilde{L}^{\infty}\big((0,T);B_{2,1}^{1+\frac{d}{2}-2\alpha}\big)} \lVert B^{(n)} \rVert_{L^{1}\big((0,T);B_{2,1}^{1+\frac{d}{2}}\big)}\\
		&+ CT^{2}\lVert\theta^{(n)} \rVert_{\tilde{L}^{\infty}\big((0,T);B_{2,1}^{1+\frac{d}{2}-\alpha}\big)}.
\end{aligned}
\end{equation*}
\end{lem}
\begin{proof}
Multiplying (\ref{eq6}) by $ 2^{(1+\frac{d}{2})j}$, summing $j $ over $j \geq 0 $ and integrating in time, we obtain
\begin{equation}\label{eq44}
\begin{aligned}
     \sum_{j\geq 0} 2^{(1+\frac{d}{2})j}\lVert \Delta_{j} u^{(n+1)}\rVert_{L^{1}(0,T,L^{2})} 
     &\leq \int_{0}^{T} \sum_{j\geq 0} 2^{(1+\frac{d}{2})j} e^{-c_{0}2^{2 \alpha j}t} \lVert \Delta_{j} u_{0}^{(n+1)}\rVert_L^{2} d\tau \\
    &+ \int_{0}^{T} \sum_{j\geq 0} 2^{(1+\frac{d}{2})j} \int_{0}^{s}e^{-c_{0}2^{2 \alpha j}(s-\tau)}(I_{1}+\cdots+I_{7}) d\tau ds.
\end{aligned}
\end{equation}
Since,
\begin{eqnarray*}\label{eq45}
    \int_{0}^{T} \sum_{j\geq 0} 2^{(1+\frac{d}{2})j} e^{-c_{0}2^{2 \alpha j}t} \lVert \Delta_{j} u_{0}^{(n+1)}\rVert_{L^{2}} dt = C \sum_{j\geq 0} 2^{(1+\frac{d}{2}-2\alpha)j} (1-e^{-c_{0}2^{2 \alpha j}T}) \lVert \Delta_{j} u_{0}^{(n+1)}\rVert_L^{2},  
\end{eqnarray*}
and since $u_{0}\in B_{2,1}^{(1+\frac{d}{2}-2\alpha}$, it follows from the dominated convergence theorem that
\begin{equation*}\label{46}
    \lim\limits_{T \rightarrow 0} \sum_{j\geq 0} 2^{(1+\frac{d}{2}-2\alpha)j} (1-e^{-c_{0}2^{2 \alpha j}T}) \lVert \Delta_{j} u_{0}^{(n+1)}\rVert_L^{2}=0.
\end{equation*}
For $j=-1$,  multiplying (\ref{eq7}) by $2^{-(1+\frac{d}{2})}$ and integrating in time, it yields
\begin{equation}
    \begin{aligned}\label{eq47}
            2^{-(1+\frac{d}{2})}\lVert \Delta_{-1} u^{(n+1)}\rVert_{L^{1}(0,T,L^{2})} &\leq 2^{-(1+\frac{d}{2})} \int_{0}^{T} \lVert \Delta_{-1} u_{0}^{(n+1)}\rVert_{L^{2}} d\tau \\
            &+  2^{-(1+\frac{d}{2})} \int_{0}^{T} \int_{0}^{s} (I_{1}+\cdots+I_{7}) d\tau ds
    \\
    &=\Bigg[2^{-2\alpha} \int_{0}^{T} 2^{-(1+\frac{d}{2}-2\alpha)} \lVert \Delta_{-1} u_{0}^{(n+1)}\rVert_{L^{2}} d\tau\\
   & +  2^{-2\alpha} \int_{0}^{T} 2^{-(1+\frac{d}{2}-2\alpha)}  \int_{0}^{s} (I_{1}+\cdots+I_{7}) d\tau ds\Bigg]. 
    \end{aligned}
\end{equation}
Clearly,
\begin{eqnarray*}\label{eq48}
    2^{-2\alpha} \int_{0}^{T} 2^{-(1+\frac{d}{2}-2\alpha)} \lVert \Delta_{-1} u_{0}^{(n+1)}\rVert_{L^{2}} d\tau \leq 2^{-2\alpha} T \lVert u_{0}^{(n+1)} \rVert_{B_{2,1}^{1+\frac{d}{2}-2\alpha}}
\end{eqnarray*}
Therefore, we can choose $T$ sufficiently small such that
\begin{equation*}\label{eq49}
       \int_{0}^{T} \sum_{j\geq 0} 2^{(1+\frac{d}{2})j} e^{-c_{0}2^{2 \alpha j}t} \lVert \Delta_{j} u_{0}^{(n+1)}\rVert_{L^{2}} dt+  2^{-2\alpha} \int_{0}^{T} 2^{-(1+\frac{d}{2}-2\alpha)} \lVert \Delta_{-1} u_{0}^{(n+1)}\rVert_{L^{2}} dt \leq \dfrac{\delta}{2}.
\end{equation*}
By Young's inequality for the time convolution, gathering (\ref{eq44}) and (\ref{eq47}) and since
\begin{equation*}\label{eq50}
     \int_{0}^{T} e^{-c_{0}2^{2\alpha j}s} ds \leq C(1-e^{-c_{0}T})2^{-2\alpha j}
\end{equation*}
we get
\begin{equation}\label{eq51}
      \lVert  u^{(n+1)}\rVert_{L^{1}(0,T,B_{2,1}^{1+\frac{d}{2}})} \leq \dfrac{\delta}{2} + C T \int_{0}^{T} \sum_{j\geq -1} 2^{(1+\frac{d}{2}-2\alpha)j}(I_{1}+\cdots+I_{7}) d\tau.
\end{equation}
We estimate the terms involving $I_{1}-I_{5}$. Arguing similarly as deriving (\ref{eq9})-(\ref{eq15})
\begin{equation}
    \begin{aligned}\label{eq52}
    C T \int_{0}^{T} \sum_{j\geq -1} 2^{(1+\frac{d}{2}-2\alpha)j}(I_{1}+\cdots+I_{7}) d\tau &\leq CT\lVert u^{(n+1)} \rVert_{\tilde{L}^{\infty}\big((0,T);B_{2,1}^{1+\frac{d}{2}-2\alpha}\big)}  \lVert u^{(n)} \rVert_{L^{1}\big((0,T);B_{2,1}^{1+\frac{d}{2}}\big)}\\
		&+CT \lVert B^{(n)} \rVert_{\tilde{L}^{\infty}\big((0,T);B_{2,1}^{1+\frac{d}{2}-2\alpha}\big)} \lVert B^{(n)} \rVert_{L^{1}\big((0,T);B_{2,1}^{1+\frac{d}{2}}\big)}\\
		&+ CT^{2}\lVert\theta^{(n)} \rVert_{\tilde{L}^{\infty}\big((0,T);B_{2,1}^{1+\frac{d}{2}-\alpha}\big)}.
\end{aligned}
\end{equation}

Inserting the estimate (\ref{eq52}) into (\ref{eq51}), we get
\begin{equation*}
    \begin{aligned}
    \lVert  u^{(n+1)}\rVert_{L^{1}(0,T,B_{2,1}^{1+\frac{d}{2}})} &\leq \dfrac{\delta}{2}
		+ CT\lVert u^{(n+1)} \rVert_{\tilde{L}^{\infty}\big((0,T);B_{2,1}^{1+\frac{d}{2}-2\alpha}\big)}  \lVert u^{(n)} \rVert_{L^{1}\big((0,T);B_{2,1}^{1+\frac{d}{2}}\big)}\\
		&+CT \lVert B^{(n)} \rVert_{\tilde{L}^{\infty}\big((0,T);B_{2,1}^{1+\frac{d}{2}-2\alpha}\big)} \lVert B^{(n)} \rVert_{L^{1}\big((0,T);B_{2,1}^{1+\frac{d}{2}}\big)}\\
		&+ CT^{2}\lVert\theta^{(n)} \rVert_{\tilde{L}^{\infty}\big((0,T);B_{2,1}^{1+\frac{d}{2}-\alpha}\big)}.
\end{aligned}
\end{equation*}

\end{proof}
\par Finally, we estimate the term $\lVert B^{(n+1)}\rVert_{L^{1}((0,T);B_{2,1}^{1+\frac{d}{2}}(\mathbb{R}^{d}))}$.
\begin{lem}\label{B2(n+1)}
The following estimate holds:
\begin{equation*}
     \lVert  B^{(n+1)}\rVert_{L^{1}(0,T,B_{2,1}^{1+\frac{d}{2}})} \leq \dfrac{\delta}{2} + CT\delta \lVert B^{(n+1)} \rVert_{\tilde{L}^{\infty}\big((0,T);B_{2,1}^{1+\frac{d}{2}-2\alpha}\big)} +CT\delta M.
\end{equation*}
\end{lem}
\begin{proof}
Multiplying (\ref{eq21}) by $2^{(1+\frac{d}{2})j}$, summing j over $j\geq 0 $ and integrating in time, we get
\begin{equation*}
    \begin{aligned}
          \sum_{j \geq0} 2^{(1+\frac{d}{2})j}\lVert B^{(n+1)} \rVert_{L^{1}(0,T;L^{2})}
          &\leq \int_{0}^{T} \sum_{j \geq0} 2^{(1+\frac{d}{2})j} e^{-c_{1}2^{2\alpha j}t} \lVert \Delta_{j} B_{0}^{(n+1)}\rVert_{L^{2}} d\tau \\
          &+ \int_{0}^{T} \sum_{j \geq0} 2^{(1+\frac{d}{2})j} \int_{0}^{s} e^{-c_{1}2^{2\alpha j}(s-\tau)} (J_{1}+\cdots+J_{6}) d\tau ds.
    \end{aligned}
\end{equation*}
Corresponding to the estimate of $\lVert u^{(n+1)}\rVert_{L^{1}((0,T);B_{2,1}^{1+\frac{d}{2}}(\mathbb{R}^{d}))}$,  we can choose $T$ sufficiently small such that
\begin{equation}\label{eq55}
    \lVert B^{(n+1)}\rVert_{L^{1}(0,T,B_{2,1}^{1+\frac{d}{2}})} \leq \dfrac{\delta}{2} + C T \int_{0}^{T} \sum_{j\geq -1} 2^{(1+\frac{d}{2}-2\alpha)j}(J_{1}+\cdots+J_{6}) d\tau. 
\end{equation}
We estimate the terms involving $J_{1}-J_{5}$. Using the same reasoning to get $(\ref{eq23})-(\ref{eq28})$
\begin{equation}\label{eq56}
        C T \int_{0}^{T} \sum_{j\geq -1} 2^{(1+\frac{d}{2}-2\alpha)j}(I_{1}+\cdots+I_{7}) d\tau \leq CT\delta  \lVert B^{(n+1)} \rVert_{\tilde{L}^{\infty}\big((0,T);B_{2,1}^{1+\frac{d}{2}-2\alpha}\big)} +CT \delta M.
\end{equation}
Putting the estimate (\ref{eq56}) into (\ref{eq55}) yields the desired outcome.
\end{proof}
\par We prove that the sequence ${(u^{(n)},v^{(n)},\theta^{(n)})}$ has a subsequence that converges to the weak solution.
\begin{lem}
 The sequence ${(u^{(n)},v^{(n)},\theta^{(n)})}$ is uniform bound in $Y$.
\end{lem}
\begin{proof}
by induction.
Recall that $(u_{0},v_{0})\in B_{2,1}^{1+\frac{d}{2}-2\alpha}(\mathbb{R}^{d}), \theta_{0}\in B_{2,1}^{1+\frac{d}{2}-\alpha}(\mathbb{R}^{d})$, according to 
\begin{eqnarray*}
    u^{(1)}=S_{1}u_{0},\ \ B^{(1)}=S_{1}B_{0},\ \ \theta^{(1)}=S_{1}\theta_{0}.
\end{eqnarray*}
Clearly
\begin{equation*}
    \begin{aligned}
          &\|u^{(1)}\|_{\tilde{L}^{\infty}(0,T;B_{2,1}^{1+\frac{d}{2}-2\alpha})}=\|S_{1}u_{0}\|_{B_{2,1}^{1+\frac{d}{2}-2\alpha}}&\leq M, \\
&\|B^{(1)}\|_{\tilde{L}^{\infty}(0,T;B_{2,1}^{1+\frac{d}{2}-2\alpha})}=\|S_{1}B_{0}\|_{B_{2,1}^{1+\frac{d}{2}-2\alpha}}&\leq M, \\
&\|\theta^{(1)}\|_{\tilde{L}^{\infty}(0,T;B_{2,1}^{1+\frac{d}{2}-\alpha})}=\|S_{1}\theta_{0}\|_{B_{2,1}^{1+\frac{d}{2}-\alpha}}&\leq M.
    \end{aligned}
\end{equation*}
 If $T>0$ is sufficiently small, then
\begin{eqnarray*}
    \begin{aligned}
        &\|u^{(1)}\|_{L^{1}(0,T;B_{2,1}^{1+\frac{d}{2}})}\leq T\|S_{1}u_{0}\|_{B_{2,1}^{1+\frac{d}{2}}}\leq TC\|u_{0}\|_{B_{2,1}^{1+\frac{d}{2}-2\alpha}}\leq\delta, \\
&\|B^{(1)}\|_{L^{1}(0,T;B_{2,1}^{1+\frac{d}{2}})}\leq T\|S_{1}B_{0}\|_{B_{2,1}^{1+\frac{d}{2}}}\leq TC\|B_{0}\|_{B_{2,1}^{1+\frac{d}{2}-2\alpha}}\leq\delta.
    \end{aligned}
\end{eqnarray*}
Assuming that ${(u^{(n)},B^{(n)},\theta^{(n)})}$ obeys the bounds defined in $Y$, namely
\begin{equation*}
    \begin{aligned}
        &\|u^{(n)}\|_{\tilde{L}^{\infty}(0,T;B_{2,1}^{1+\frac{d}{2}-2\alpha})}\leq M, \quad \|B^{(n)}\|_{\tilde{L}^{\infty}(0,T;B_{2,1}^{1+\frac{d}{2}-2\alpha})}\leq M, \quad \|\theta^{(n)}\|_{\tilde{L}^{\infty}(0,T;B_{2,1}^{1+\frac{d}{2}-\alpha})}\leq M,\\
        &\|u^{(n)}\|_{L^{1}(0,T;B_{2,1}^{1+\frac{d}{2}})}\leq\delta \quad
        \|B^{(n)}\|_{L^{1}(0,T;B_{2,1}^{1+\frac{d}{2}})}\leq\delta.
    \end{aligned}
\end{equation*}
We demonstrate that ${(u^{(n+1)},v^{(n+1)},\theta^{(n+1)})}$ fulfills the equivalent bounds for the aforementioned $T>0$ and $M>0$, specifically
\begin{equation*}
    \begin{aligned}
        &\|u^{(n+1)}\|_{\tilde{L}^{\infty}(0,T;B_{2,1}^{1+\frac{d}{2}-2\alpha})}\leq M,\\  
				&\|B^{(n+1)}\|_{\tilde{L}^{\infty}(0,T;B_{2,1}^{1+\frac{d}{2}-2\alpha})}\leq M, \\
				&\|\theta^{(n+1)}\|_{\tilde{L}^{\infty}(0,T;B_{2,1}^{{1+\frac{d}{2}-\alpha}})}\leq M,\\
&\|u^{(n+1)}\|_{L^{1}(0,T;B_{2,1}^{1+\frac{d}{2}})}\leq\delta,\\
&\|B^{(n+1)}\|_{L^{1}(0,T;B_{2,1}^{1+\frac{d}{2}})}\leq\delta.
    \end{aligned}
\end{equation*}
Lemma \ref{u(n+1)} gives 
\begin{equation}\label{eq16}
     \lVert u^{(n+1)} \rVert_{\tilde{L}^{\infty}\big((0,T);B_{2,1}^{1+\frac{d}{2}-2\alpha}\big)} \leq \lVert u_{0}^{(n+1)} \rVert_{B_{2,1}^{1+\frac{d}{2}-2\alpha}} + C \delta \lVert u^{(n+1)} \rVert_{\tilde{L}^{\infty}\big((0,T);B_{2,1}^{1+\frac{d}{2}-2\alpha}\big)} + C\delta M +CTM.
\end{equation}
Collecting the estimates (\ref{eq23})-(\ref{eq28}) and inserting them into (\ref{eq22}), we get
\begin{equation}\label{eq29}
     \lVert B^{(n+1)} \rVert_{\tilde{L}^{\infty}\big((0,T);B_{2,1}^{1+\frac{d}{2}-2\alpha}\big)} \leq \lVert B_{0}^{(n+1)} \rVert_{B_{2,1}^{1+\frac{d}{2}-2\alpha}} +  C \delta \lVert B^{(n+1)} \rVert_{\tilde{L}^{\infty}\big((0,T);B_{2,1}^{1+\frac{d}{2}-2\alpha}\big)} + C\delta M.
\end{equation}
Lemma \ref{estimatetheta(n+1)} states that by adding estimates (\ref{eq38})-(\ref{eq40}) to (\ref{eq37}), we get
\begin{equation}\label{eq41}
    \begin{aligned}
          \lVert \theta^{(n+1)} \rVert_{\tilde{L}^{\infty}\big((0,T);B_{2,1}^{1+\frac{d}{2}-\alpha}\big)} &\leq \lVert \theta_{0}^{(n+1)} \rVert_{B_{2,1}^{1+\frac{d}{2}-\alpha}}+C\delta  \lVert \theta^{(n+1)} \rVert_{\tilde{L}^{\infty}\big((0,T);B_{2,1}^{1+\frac{d}{2}-\alpha}\big)} +C\delta+C\delta M.
    \end{aligned}
\end{equation}
Thus, combing the estimates (\ref{eq16}), (\ref{eq29}) and (\ref{eq41}), we obtain
  \begin{equation}\label{eq42}
      \begin{aligned}
\lVert (u^{(n+1)},B^{(n+1)})\rVert_{\tilde{L}^{\infty}(0,T;B_{2,1}^{1+\frac{d}{2}-2\alpha})}
+\lVert\theta^{(n+1)}\rVert_{\tilde{L}^{\infty}(0,T;B_{2,1}^{1+\frac{d}{2}-\alpha})}
&\leq\|(u_{0}^{(n+1)},B_{0}^{(n+1)})\|_{B_{2,1}^{1+\frac{d}{2}-2\alpha}}+\|\theta_{0}^{(n+1)}\|_{B_{2,1}^{1+\frac{d}{2}-\alpha}} \\
& + C\delta\|(u^{(n+1)},B^{(n+1)})\|_{\tilde{L}^{\infty}(0,T;B_{2,1}^{1+\frac{d}{2}-2\alpha})}\\
&+C\delta\|\theta^{(n+1)}\|_{\tilde{L}^{\infty}(0,T;B_{2,1}^{1+\frac{d}{2}-\alpha})}\\
&+C\delta M+C\delta+CTM.
 \end{aligned}
 \end{equation}
Choosing $C\delta\leq \min(\frac{1}{8},\frac{M}{8})$ and $CT\leq\frac{1}{8}$, we have
\begin{equation}
    \begin{aligned}\label{eq43}
     \|(u^{(n+1)},B^{(n+1)})\|_{\tilde{L}^{\infty}(0,T;B_{2,1}^{1+\frac{d}{2}-2\alpha})}
+\|\theta^{(n+1)}\|_{\tilde{L}^{\infty}(0,T;B_{2,1}^{1+\frac{d}{2}-\alpha})}
&\leq\frac{M}{2}+\frac{1}{8}\|(u^{(n+1)},B^{(n+1)})\|_{\tilde{L}^{\infty}(0,T;B_{2,1}^{1+\frac{d}{2}-2\alpha})}\\
&+\frac{1}{8}\|\theta^{(n+1)}\|_{\tilde{L}^{\infty}(0,T;B_{2,1}^{1+\frac{d}{2}-\alpha})}\\
&+\frac{3M}{8}.
    \end{aligned}
\end{equation}
It is evident through simplification that
\begin{align}
\nonumber\|(u^{(n+1)},B^{(n+1)})\|_{\tilde{L}^{\infty}(0,T;B_{2,1}^{1+\frac{d}{2}-2\alpha})}\leq M, \ \ \ \ \|\theta^{(n+1)}\|_{\tilde{L}^{\infty}(0,T;B_{2,1}^{1+\frac{d}{2}-\alpha})}\leq M.
\end{align}
It is implied by the property of Chemin-Lerner type Besov spaces (see definition (\ref{def1})):
\begin{align}
\nonumber&\|(u^{(n+1)},B^{(n+1)})\|_{L^{\infty}(0,T;B_{2,1}^{1+\frac{d}{2}-2\alpha})}\leq M, \ \ \ \ \|\theta^{(n+1)}\|_{L^{\infty}(0,T;B_{2,1}^{1+\frac{d}{2}-\alpha})}\leq M.
\end{align}
According to the Lemma \ref{u2(n+1)} we get
\begin{eqnarray*}\label{eq53}
    \lVert  u^{(n+1)}\rVert_{L^{1}(0,T,B_{2,1}^{1+\frac{d}{2}})} \leq \dfrac{\delta}{2} + CT \delta M+ CT^{2}M.
\end{eqnarray*}
Choosing $T$ sufficiently small such that $CT \leq \min(\dfrac{1}{4M},\dfrac{\delta}{4TM})$ it holds that
\begin{equation*}\label{eq54}
     \lVert  u^{(n+1)}\rVert_{L^{1}(0,T,B_{2,1}^{1+\frac{d}{2}})} \leq \dfrac{\delta}{2} +\dfrac{\delta}{4}+\dfrac{\delta}{4}=\delta.
\end{equation*}
Lemma \ref{B2(n+1)} gives
\begin{equation*}
     \lVert  B^{(n+1)}\rVert_{L^{1}(0,T,B_{2,1}^{1+\frac{d}{2}})} \leq \dfrac{\delta}{2} + CT \delta M.
\end{equation*}
Choosing $T$ sufficiently small such that $CT \leq \dfrac{1}{2M}$ it holds that
\begin{equation*}
    \lVert B^{(n+1)}\rVert_{L^{1}(0,T,B_{2,1}^{1+\frac{d}{2}})} \leq \dfrac{\delta}{2} +\dfrac{\delta}{2}=\delta.
\end{equation*}
\end{proof}
\section{Proof of the existence part}
Denote by $\stackrel{\ast}{\rightharpoonup}$ the $weak^{*}$ convergence. Using the uniform bounds above, we can extract a weakly convergent subsequence based on $T$. There exists $(u,B,\theta) \in Y $ such that a subsequence of $(u^{(n)},B^{(n)},\theta^{(n)})  $(still denoted by $(u^{(n)},B^{(n)},\theta^{(n)})$) satisfies:
\begin{eqnarray*}
    (u^{(n)},B^{(n)})\stackrel{\ast}{\rightharpoonup} (u,B) \;\;\mbox{ in }\;\; L^{1}(0,T,B_{2,1}^{1+\frac{d}{2}}) \cap L^{\infty}\big((0,T);B_{2,1}^{1+\frac{d}{2}-2\alpha}\big),
    \\
    \theta^{(n)}\stackrel{\ast}{\rightharpoonup} \theta \;\;\mbox{ in }\;\ L^{\infty}\big((0,T);B_{2,1}^{1+\frac{d}{2}-\alpha}\big)
\end{eqnarray*}
Furthermore, we may demonstrate by using the equation
(\ref{sys2}) that $(\partial_{t}u^{n},\partial_{t}B^{n},\partial_{t}\theta^{n})$ is uniformly bounded
\\
\textbf{$$(\partial_{t} u^{n},\partial_{t} B^{n})\in L^{1}\big((0,T);B_{2,1}^{1+\frac{d}{2}-2\alpha}\big).
$$ }
\begin{proof}
Using definition (\ref{def1}) and according to (\ref{sys2})$_{1}$, we get
\begin{equation*}
    \begin{aligned}
         \int_{0}^{T} \lVert \partial_{t} u^{(n+1)}\rVert_{B_{2,1}^{1+\frac{d}{2}-2\alpha}} d\tau 
        &\leq  \int_{0}^{T} \lVert (-\Delta)^{\alpha}u^{(n+1)}\rVert_{B_{2,1}^{1+\frac{d}{2}-2\alpha}} d\tau 
         + \int_{0}^{T} \lVert u^{(n)}\cdot \nabla u^{(n+1)}\rVert_{B_{2,1}^{1+\frac{d}{2}-2\alpha}} d\tau\\
         &+\int_{0}^{T} \lVert B^{(n)}\cdot \nabla B^{(n+1)}\rVert_{B_{2,1}^{1+\frac{d}{2}-2\alpha}} d\tau  +\int_{0}^{T} \lVert \theta^{(n)} e_{d}\rVert_{B_{2,1}^{1+\frac{d}{2}-2\alpha}} d\tau. 
    \end{aligned}
\end{equation*}
Applying Bernstein's inequality (see lemma \ref{lem1}) to the $\Delta_{j}u^{(n+1)}$, multiplying by $2^{(1+\frac{d}{2}-2\alpha)j}$ and summing over $j$, choosing $L^{1}(0,T)$-norm, we derive
\begin{equation}\label{eq57}
    \begin{aligned}
        \int_{0}^{T} \lVert (-\Delta)^{\alpha}u^{(n+1)}\rVert_{B_{2,1}^{1+\frac{d}{2}-2\alpha}} d\tau &\leq C_{1} \int_{0}^{T} \lVert u^{(n+1)}\rVert_{B_{2,1}^{1+\frac{d}{2}}} d\tau \\
        &= C_{1}\lVert u^{(n+1)}\rVert_{L^{1}{(0,T;B_{2,1}^{1+\frac{d}{2}})}}\\
        &\leq C_{1} \delta.
    \end{aligned}
\end{equation}
The estimation of the right-hand side terms is similar; we use an example to provide proof.
\begin{equation*}
    \begin{aligned}
         \int_{0}^{T} \lVert u^{(n)}\cdot \nabla u^{(n+1)}\rVert_{B_{2,1}^{1+\frac{d}{2}-2\alpha}} d\tau&\leq \int_{0}^{T} \sum_{m\geq -1}2^{(1+\frac{d}{2}-2\alpha)j}\Bigg[\sum_{m\leq j-1}2^{j}2^{\frac{d}{2}m} \lVert \Delta_{m} u^{(n)} \rVert _{L^{2}}\sum_{|j-k|\leq 2}\lVert \Delta_{k}u^{(n+1)} \rVert _{L^{2}} \\ &+
    \sum_{|j-k|\leq 2}\lVert \Delta_{k} u^{(n)} \rVert _{L^{2}}  \sum_{m\leq j-1}2^{(1+\frac{d}{2})m} \lVert \Delta_{m} u^{(n+1)} \rVert_{L^{2}}
   \\ &+
    \sum_{k\geq j-4} 2^{j}2^{\frac{d}{2}k}
\lVert \Delta_{k} u^{(n)} \rVert _{L^{2}}\lVert \tilde{\Delta}_{k} u^{(n+1)} \rVert _{L^{2}}  \Bigg]^{2} d\tau.
    \end{aligned}
\end{equation*}
The procedure is analogous to the estimates (\ref{eq9})-(\ref{eq15}) and the property of Chemin-Lerner type Besov spaces (see definition (\ref{def1})). So,
\begin{equation}\label{eq58}
\begin{aligned}
     \int_{0}^{T} \lVert u^{(n)}\cdot \nabla u^{(n+1)}\rVert_{B_{2,1}^{1+\frac{d}{2}-2\alpha}} d\tau  
     &\leq 
    C \lVert u^{(n)} \rVert_{L^{\infty}\big((0,T);B_{2,1}^{1+\frac{d}{2}-2\alpha})} \lVert u^{(n+1)} \rVert_{L^{1}\big((0,T);B_{2,1}^{1+\frac{d}{2}})}\\
    &\leq C \delta M.
\end{aligned}
Also,
\end{equation}
\begin{equation}\label{eq59}
    \begin{aligned}
     \int_{0}^{T} \lVert B^{(n)}\cdot \nabla B^{(n+1)}\rVert_{B_{2,1}^{1+\frac{d}{2}-2\alpha}} d\tau 
     &\leq 
    C \lVert B^{(n)} \rVert_{L^{\infty}\big((0,T);B_{2,1}^{1+\frac{d}{2}-2\alpha})} \lVert B^{(n+1)} \rVert_{L^{1}\big((0,T);B_{2,1}^{1+\frac{d}{2}})}\\
    &\leq C \delta M
\end{aligned}
\end{equation}
and
  \begin{equation}\label{eq60}
       \begin{aligned}
 \int_{0}^{T} \lVert \theta^{(n)} e_{d}\rVert_{B_{2,1}^{1+\frac{d}{2}-2\alpha}} d\tau 
 &\leq \int_{0}^{T} \lVert \theta^{(n)}\rVert_{B_{2,1}^{1+\frac{d}{2}-2\alpha}} d\tau \\
 &\leq C T \lVert \theta^{(n)}\rVert_{L^{\infty}\big((0,T);B_{2,1}^{1+\frac{d}{2}-2\alpha})}\\
 &\leq CT M.
       \end{aligned}
   \end{equation}
By combining the estimations (\ref{eq57}) and (\ref{eq60}), we get
\begin{equation}\label{eq61}
    \begin{aligned}
        \int_{0}^{T} \lVert \partial_{t} u^{(n+1)}\rVert_{B_{2,1}^{1+\frac{d}{2}-2\alpha}} d\tau 
        &\leq C_{1}\delta+C\delta M+ CTM < \infty.
    \end{aligned}
\end{equation}
The definition (\ref{def1}) and the equation (\ref{sys2})$_{2}$ give
\begin{equation*}
    \begin{aligned}
    \int_{0}^{T} \lVert \partial_{t} B^{(n+1)}\rVert_{B_{2,1}^{1+\frac{d}{2}-2\alpha}} d\tau  
    &\leq  \int_{0}^{T} \lVert (-\Delta)^{\alpha}B^{(n+1)}\rVert_{B_{2,1}^{1+\frac{d}{2}-2\alpha}} d\tau\\
    &+ \int_{0}^{T} \lVert u^{(n)}\cdot \nabla B^{(n+1)}\rVert_{B_{2,1}^{1+\frac{d}{2}-2\alpha}} d\tau+
    \int_{0}^{T} \lVert B^{(n)}\cdot \nabla u^{(n)}\rVert_{B_{2,1}^{1+\frac{d}{2}-2\alpha}} d\tau.
    \end{aligned}
\end{equation*}
By implementing Bernstein's inequality (see lemma \ref{lem1}) to the $\Delta_{j}B^{(n+1)}$, multiplying by $2^{(1+\frac{d}{2}-2\alpha)j}$ and summing over $j$, we establish
\begin{equation}\label{eq62}
    \begin{aligned}
        \int_{0}^{T} \lVert (-\Delta)^{\alpha}B^{(n+1)}\rVert_{B_{2,1}^{1+\frac{d}{2}-2\alpha}} d\tau &\leq C_{1} \int_{0}^{T} \lVert B^{(n+1)}\rVert_{B_{2,1}^{1+\frac{d}{2}}} d\tau \\
        &= C_{1}\lVert B^{(n+1)}\rVert_{L^{1}{(0,T;B_{2,1}^{1+\frac{d}{2}})}}\\
        &\leq C_{1} \delta
    \end{aligned}
\end{equation}
The process is analogous to the estimates (\ref{eq23})-(\ref{eq28}) and the property of Chemin-Lerner type Besov spaces (see definition (\ref{def1})). We obtain
\begin{equation}\label{eq63}
    \begin{aligned}
     \int_{0}^{T} \lVert u^{(n)}\cdot \nabla B^{(n+1)}\rVert_{B_{2,1}^{1+\frac{d}{2}-2\alpha}} d\tau 
     &\leq 
    C \lVert u^{(n)} \rVert_{L^{\infty}\big((0,T);B_{2,1}^{1+\frac{d}{2}-2\alpha})} \lVert B^{(n+1)} \rVert_{L^{1}\big((0,T);B_{2,1}^{1+\frac{d}{2}})}\\
    &\leq C \delta M.
\end{aligned}
\end{equation}
\begin{equation}\label{eq64}
    \begin{aligned}
     \int_{0}^{T} \lVert B^{(n)}\cdot \nabla u^{(n)}\rVert_{B_{2,1}^{1+\frac{d}{2}-2\alpha}} d\tau 
     &\leq 
    C \lVert B^{(n)} \rVert_{L^{\infty}\big((0,T);B_{2,1}^{1+\frac{d}{2}-2\alpha})} \lVert u^{(n)} \rVert_{L^{1}\big((0,T);B_{2,1}^{1+\frac{d}{2}})}\\
    &\leq C \delta M.
\end{aligned}
\end{equation}
Finally, combing the estimates (3.17), (3.28) and (3.38), we get
\begin{equation}\label{eq65}
    \begin{aligned}
        \int_{0}^{T} \lVert \partial_{t} B^{(n+1)}\rVert_{B_{2,1}^{1+\frac{d}{2}-2\alpha}} d\tau 
        &\leq C_{1}\delta+C\delta M < \infty.
    \end{aligned}
\end{equation}
Consequently,
$$
\partial_{t} \theta^{n}\in L^{2}\big((0,T);B_{2,1}^{1+\frac{d}{2}-2\alpha}\big).
$$
\begin{proof}
Definition (\ref{def1}) and  equation (\ref{sys2})$_{3}$ give
\begin{equation}\label{eq66}
      \int_{0}^{T} \lVert \partial_{t} \theta^{(n+1)}\rVert_{B_{2,1}^{1+\frac{d}{2}-2\alpha}}^{2} d\tau \leq  \int_{0}^{T} \lVert u^{(n)}\cdot \nabla \theta^{(n+1)}\rVert_{B_{2,1}^{1+\frac{d}{2}-2\alpha}}^{2} d\tau +    \int_{0}^{T} \lVert u^{(n)}\cdot e_{d}\rVert_{B_{2,1}^{1+\frac{d}{2}-2\alpha}}^{2} d\tau,
\end{equation}
and
\begin{equation}\label{eq67}
     \begin{aligned}
        \int_{0}^{T} \lVert u^{(n)}\cdot e_{d}\rVert_{B_{2,1}^{1+\frac{d}{2}-2\alpha}}^{2} d\tau 
    &\leq \int_{0}^{T} \lVert u^{(n)}\rVert_{B_{2,1}^{1+\frac{d}{2}-2\alpha}}^{2} d\tau  \\
    &\leq C T \lVert u^{(n)}\rVert_{\widetilde{L}^{\infty}\big((0,T);B_{2,1}^{1+\frac{d}{2}-2\alpha})}^{2}\\
    &\leq CT M^{2}.
    \end{aligned}
\end{equation}
    Now, we estiblish the estimate of
    \begin{equation}\label{eq68}
        \begin{aligned}
               \int_{0}^{t} \lVert u^{(n)}\cdot \nabla \theta^{(n+1)}\rVert_{B_{2,1}^{1+\frac{d}{2}-2\alpha}}^{2} d\tau 
    &\leq  \int_{0}^{T} \Bigg[\sum_{m\geq -1}2^{(1+\frac{d}{2}-2\alpha)j}\big(\sum_{m\leq j-1}2^{j}2^{\frac{d}{2}m} \lVert \Delta_{m} u^{(n)} \rVert _{L^{2}}\sum_{|j-k|\leq 2}\lVert \Delta_{k}\theta^{(n+1)} \rVert _{L^{2}} \\
    & +  \sum_{|j-k|\leq 2}\lVert \Delta_{k} u^{(n)} \rVert _{L^{2}}  \sum_{m\leq j-1}2^{(1+\frac{d}{2})m} \lVert \Delta_{m} \theta^{(n+1)} \rVert_{L^{2}}
   \\
   & + \sum_{k\geq j-4} 2^{j}2^{\frac{d}{2}k}
\lVert \Delta_{k} u^{(n)} \rVert _{L^{2}}\lVert \tilde{\Delta}_{k} \theta^{(n+1)} \rVert _{L^{2}} \big) \Bigg]^{2} d\tau.
        \end{aligned}
    \end{equation}
Note that the terms on the right hand side can be bounded respectively by
\begin{equation*}
    \begin{aligned}
         &  R_{1}= \int_{0}^{T} \bigg(\sum_{m\geq -1}2^{(1+\frac{d}{2}-2\alpha)j}\sum_{m\leq j-1}2^{j}2^{\frac{d}{2}m} \lVert \Delta_{m} u^{(n)} \rVert _{L^{2}}\sum_{|j-k|\leq 2}\lVert \Delta_{k}\theta^{(n+1)} \rVert _{L^{2}}  \bigg)^{2} d\tau 
   \\
& R_{2}= \int_{0}^{T} \bigg(\sum_{m\geq -1}2^{(1+\frac{d}{2}-2\alpha)j}  \sum_{|j-k|\leq 2}\lVert \Delta_{k} u^{(n)} \rVert _{L^{2}}  \sum_{m\leq j-1}2^{(1+\frac{d}{2})m} \lVert \Delta_{m} \theta^{(n+1)} \rVert_{L^{2}}\bigg)^{2} d\tau 
  \\
& R_{3}= \int_{0}^{T} \bigg(\sum_{m\geq -1}2^{(1+\frac{d}{2}-2\alpha)j} \sum_{k\geq j-4} 2^{j}2^{\frac{d}{2}k}
\lVert \Delta_{k} u^{(n)} \rVert _{L^{2}}\lVert \tilde{\Delta}_{k} \theta^{(n+1)} \rVert _{L^{2}} \bigg)^{2} d\tau. 
    \end{aligned}
\end{equation*}
Also,
    \begin{equation}\label{eq69}
        \begin{aligned}
   R_{1} &\leq \int_{0}^{T}\big(\sum_{m\geq -1}\sum_{m\leq j-1}2^{(\alpha-1)(m-j)}2^{(1+\frac{d}{2}-\alpha)j}\lVert \Delta_{j}\theta^{(n+1)} \rVert _{L^{2}}2^{(1+\frac{d}{2}-\alpha)m}\lVert \Delta_{m} u^{(n)} \rVert _{L^{2}}\big)^{2}d\tau \\
  &\leq
   C \lVert \theta^{(n+1)} \rVert_{\tilde{L}^{\infty}\big((0,T);B_{2,1}^{1+\frac{d}{2}-\alpha})}^{2} \lVert u^{(n)} \rVert_{\tilde{L}^{2}\big((0,T);B_{2,1}^{1+\frac{d}{2}-\alpha})}^{2}
    \end{aligned}
    \end{equation}
		and
      \begin{equation}\label{eq70}
        \begin{aligned}
   R_{2} &\leq \int_{0}^{T}\big(\sum_{m\geq -1}\sum_{m\leq j-1}2^{\alpha(m-j)}2^{(1+\frac{d}{2}-\alpha)j}\lVert \Delta_{j}u^{(n)} \rVert _{L^{2}}2^{(1+\frac{d}{2}-\alpha)m}\lVert \Delta_{m} \theta^{(n+1)} \rVert _{L^{2}}\big)^{2}d\tau \\
  &\leq
   C \lVert \theta^{(n+1)} \rVert_{\tilde{L}^{\infty}\big((0,T);B_{2,1}^{1+\frac{d}{2}-\alpha})}^{2} \lVert u^{(n)} \rVert_{\tilde{L}^{2}\big((0,T);B_{2,1}^{1+\frac{d}{2}-\alpha})}^{2}
   \end{aligned}
    \end{equation}
     \begin{equation}\label{eq71}
        \begin{aligned}
   R_{3} &\leq \int_{0}^{T}\big(\sum_{m\geq -1}\sum_{k\geq j-4}2^{(2+\frac{d}{2}-2\alpha)(j-k)}2^{(1+\frac{d}{2}-\alpha)k}\lVert \Delta_{k}u^{(n)} \rVert _{L^{2}}2^{(1+\frac{d}{2}-\alpha)k}\lVert \tilde{\Delta}_{k} \theta^{(n+1)} \rVert _{L^{2}}\big)^{2}d\tau \\
  &\leq
   C \lVert \theta^{(n+1)} \rVert_{\tilde{L}^{\infty}\big((0,T);B_{2,1}^{1+\frac{d}{2}-\alpha})}^{2} \lVert u^{(n)} \rVert_{\tilde{L}^{2}\big((0,T);B_{2,1}^{1+\frac{d}{2}-\alpha})}^{2}\;\;\mbox{with}\;\;\alpha < 1+\dfrac{d}{4}.
        \end{aligned}
    \end{equation}
Since
$ u \in L^{\infty}\big((0,T);B_{2,1}^{1+\frac{d}{2}-2\alpha}) \cap L^{1}\big((0,T);B_{2,1}^{1+\frac{d}{2}})$, by a simple interpolation inequality (see for example \cite{Tr}). Also, $$ u\in L^{2}\big((0,T);B_{2,1}^{1+\frac{d}{2}-\alpha}).$$ Therefore, combining with (\ref{eq69})-(\ref{eq71}) and inserting them into (\ref{eq68}) we obtain
\begin{equation*}
    \begin{aligned}
    \int_{0}^{T} \lVert \partial_{t} \theta^{(n+1)}\rVert_{B_{2,1}^{1+\frac{d}{2}-2\alpha}}^{2} d\tau
    &\leq
    \int_{0}^{T} \lVert u^{(n)}\cdot \nabla \theta^{(n+1)}\rVert_{B_{2,1}^{1+\frac{d}{2}-2\alpha}}^{2} d\tau +    \int_{0}^{T} \lVert u^{(n)}\cdot e_{d}\rVert_{B_{2,1}^{1+\frac{d}{2}-2\alpha}}^{2} d\tau  \\
    &\leq 
    C \lVert \theta^{(n+1)} \rVert_{\tilde{L}^{\infty}\big((0,T);B_{2,1}^{1+\frac{d}{2}-\alpha})}^{2} \lVert u^{(n)} \rVert_{\tilde{L}^{2}\big((0,T);B_{2,1}^{1+\frac{d}{2}-\alpha})}^{2}+ CTM^{2} \\ &\leq 
    C\delta^{2} M^{2}+ CTM^{2} \\
    & < \infty.
\end{aligned}
\end{equation*}
\end{proof}
\end{proof}
To sum up, we prove that $(\partial_{t} u^{n},\partial_{t}  B^{n},\partial_{t}  \theta^{n})$ is uniformly bounded
\begin{equation*}
    \begin{aligned}
   & \partial_{t} u^{n} \in L^{1}\big((0,T);B_{2,1}^{1+\frac{d}{2}-2\alpha}\big), 
     \\
     & \partial_{t}  B^{n} \in L^{1}\big((0,T);B_{2,1}^{1+\frac{d}{2}-2\alpha}\big), 
     \\
     &\partial_{t}  \theta^{n} \in L^{2}\big((0,T);B_{2,1}^{1+\frac{d}{2}-2\alpha}\big).
    \end{aligned}
\end{equation*}
Let $\{\chi_{j}\}_{j\in \mathbb{N}}$
be a sequence of smooth functions with value in $[0,1]$ supported in the ball
$B(0,j+1)$ . For any $j\in\mathbb{Z}$, we have
\begin{eqnarray*}
    \chi_{j}u^{(n)}\in L^{1}(0,T,B_{2,1}^{1+\frac{d}{2}}) \cap L^{\infty}\big((0,T);B_{2,1}^{1+\frac{d}{2}-2\alpha}\big) , \partial_{t}(  \chi_{j}u^{(n)})\in L^{1}\big((0,T);B_{2,1}^{1+\frac{d}{2}-2\alpha}\big),
\\
    \chi_{j}B^{(n)}\in L^{1}(0,T,B_{2,1}^{1+\frac{d}{2}}) \cap L^{\infty}\big((0,T);B_{2,1}^{1+\frac{d}{2}-2\alpha}\big) , \partial_{t}(  \chi_{j}B^{(n)})\in L^{1}\big((0,T);B_{2,1}^{1+\frac{d}{2}-2\alpha}\big),
    \\
     \chi_{j}\theta^{(n)}\in L^{\infty}\big((0,T);B_{2,1}^{1+\frac{d}{2}-2\alpha}\big) , \partial_{t}(  \chi_{j}\theta^{(n)})\in L^{2}\big((0,T);B_{2,1}^{1+\frac{d}{2}-2\alpha}\big).
\end{eqnarray*}
Utilizing the Aubin-Lions Lemma and the Aubin-Lions-Simon Lemma, we obtain
\begin{eqnarray}
     \chi_{j}u^{(n)}\rightarrow u \in L^{2}\big((0,T);B_{2,1}^{1+\frac{d}{2}-s_{1}}\big) \;\;\mbox{ for }\;\; s_{1}\in(0,2\alpha),
     \\
      \chi_{j}B^{(n)}\rightarrow B \in L^{2}\big((0,T);B_{2,1}^{1+\frac{d}{2}-s_{1}}\big) \;\;\mbox{ for }\;\; s_{2}\in(0,2\alpha),
     \\
      \chi_{j}\theta^{(n)}\rightarrow \theta \in L^{2}\big((0,T);B_{2,1}^{1+\frac{d}{2}-s_{3}}\big) \;\;\mbox{ for }\;\; s_{3}\in(\alpha,2\alpha).
\end{eqnarray}
Given that we are dealing with the entire space $\mathbb{R}^{d}$, it is necessary to combine Cantor's diagonal process to demonstrate that a subsequence of the weakly convergent sequence $( \chi_{j}u^{(n)},\chi_{j}B^{(n)},\chi_{j}\theta^{(n)})$ exists such that:
\begin{eqnarray}
    ( \chi_{j}u^{(n)},\chi_{j}B^{(n)},\chi_{j}\theta^{(n)}) \rightarrow (u,B,\theta) \;\;\mbox{ for }\;\; j \rightarrow 0.
\end{eqnarray}
As a result, we have the following strongly convergent property:
\begin{eqnarray}
     u^{(n)}\rightarrow u \in L^{2}\big((0,T);B_{2,1}^{1+\frac{d}{2}-s_{1}}\big) \;\;\mbox{ for }\;\; s_{1}\in(0,2\alpha),
     \\
     B^{(n)}\rightarrow B \in L^{2}\big((0,T);B_{2,1}^{1+\frac{d}{2}-s_{1}}\big) \;\;\mbox{ for }\;\; s_{2}\in(0,2\alpha),
     \\
     \theta^{(n)}\rightarrow \theta \in L^{2}\big((0,T);B_{2,1}^{1+\frac{d}{2}-s_{3}}\big) \;\;\mbox{ for }\;\; s_{3}\in(\alpha,2\alpha).
\end{eqnarray}
\section{Proof for the uniqueness part of theorem 1.1}
This section aims to prove the uniqueness part of Theorem 1.1. To this end, assume that $(u^{1},B^{1},\theta^{1})$ and $(u^{2},B^{2},\theta^{2})$ are two solutions of (\ref{sys1}). Their difference
$(\tilde{u},\Tilde{B},\Tilde{\theta})$ with
$$ \tilde{u}=u^{2}-u^{1}, ~~~~~~ \tilde{B}=B^{2}-B^{1}, ~~~~~~ \tilde{\theta}=\theta^{2}-\theta^{1}
$$
satisfies 
\begin{equation}\label{eq72}
\left\{
    \begin{aligned}
         & \partial_{t} \widetilde{u} +\mu (-\Delta)^{\alpha} \widetilde{u}  = -\mathbb{P}\bigg[u^{(1)}\cdot\nabla \widetilde{u} + \widetilde{u}\cdot \nabla u^{(2)}- B^{(1)}\cdot\nabla \widetilde{B} - \widetilde{B}\cdot \nabla B^{(2)}-\widetilde{\theta} e_{d} \bigg],
     \\
    & \partial_{t} \widetilde{B} +\eta (-\Delta)^{\beta}\widetilde{B} = -u^{(1)}\cdot \nabla \widetilde{B}-\widetilde{u}\cdot \nabla B^{(2)} + B^{(1)}\cdot \nabla \widetilde{u}+ \widetilde{B} \cdot \nabla u^{(2)},
     \\
   & \partial_{t} \widetilde{\theta} =- u^{(1)}\cdot \nabla \widetilde{\theta}-\widetilde{u}\cdot \nabla \theta^{(2)}+  \widetilde{u}\cdot e_{d},
  \\
&\Div{\widetilde{u}}=\Div{\widetilde{B}}=0,
\\
&(\widetilde{u},\widetilde{B},\widetilde{\theta} )\mid_{t=0} = 0,
    \end{aligned}
    \right.
\end{equation}
\begin{proof}
Let $j \geq 0 $ be an integer, applying $\Delta_{j} $ to (\ref{eq72})$_{1}$, and dotting the equation by $\Delta_{j} \widetilde{u} $ we find
\begin{equation*}
    \begin{aligned}
           \dfrac{1}{2} \dfrac{d}{dt}\lVert\Delta_{j} \widetilde{u}||_{L^{2}}^{2} +\mu \rVert \Lambda^{\alpha} \Delta_{j} \widetilde{u} \rVert_{L^{2}}^{2} &= - \int_{\mathbb{R}^{d}} \Delta_{j}(u^{(1)}\cdot \nabla \widetilde{u})\cdot \Delta_{j}\widetilde{u} dx - \int_{\mathbb{R}^{d}} \Delta_{j}(\widetilde{u}\cdot \nabla u^{(2)})\cdot \Delta_{j}\widetilde{u} dx
    \\
   & + \int_{\mathbb{R}^{d}} \Delta_{j}(B^{(1)}\cdot \nabla \widetilde{B})\cdot \Delta_{j}\widetilde{u} dx + \int_{\mathbb{R}^{d}} \Delta_{j}(\widetilde{B}\cdot \nabla B^{(2)})\cdot \Delta_{j}\widetilde{u} dx \\
   &+ \int_{\mathbb{R}^{d}} \Delta_{j} \widetilde{\theta} \cdot \Delta_{j} \widetilde{u} dx.
    \end{aligned}
\end{equation*}
By integrating in time and eliminating $\lVert\Delta_{j} \widetilde{u}||_{L^{2}}^{2} $ from both sides of the inequality, based on Lemma (\ref{lem2}), we can achieve the following:
\begin{equation}\label{eq73}
       \lVert\Delta_{j} \widetilde{u}||_{L^{2}}^{2} \leq
    C \int_{0}^{t} e^{-c_{0}2^{2 \alpha j}(t-\tau)}\bigg(E_{1}+E_{2}+ \lVert \Delta_{j}(\widetilde{u} \cdot \nabla u^{(2)} \rVert_{L^{2}}+ \lVert \Delta_{j}(\widetilde{B} \cdot \nabla B^{(2)} \rVert_{L^{2}} + \lVert \Delta_{j} \widetilde{\theta}\rVert_{L^{2}}\bigg),
\end{equation}
where
\begin{equation*}
    \begin{aligned}
        E_{1} &:= C \bigg( \lVert \Delta_{j} u^{(1)} \rVert_{L^{2}} \sum_{m \leq j-1} 2^{(1+\frac{d}{2})m} \lVert \Delta_{m} \widetilde{u} \rVert_{L^{2}}\\
        &+\lVert \Delta_{j} \widetilde{u} \rVert_{L^{2}} \sum_{m \leq j-1} 2^{(1+\frac{d}{2})m}\lVert \Delta_{m} u^{(1)} \rVert_{L^{2}}\\
        &+2^{(1+\frac{d}{2})j} \sum_{k \geq j-4} \lVert \Delta_{k} u^{(1)} \rVert_{L^{2}}\lVert \widetilde{\Delta}_{k} \widetilde{u} \rVert_{L^{2}}\bigg) 
    \end{aligned}
		and
\end{equation*}
\begin{equation*}
    \begin{aligned}
        E_{2} &:= C\bigg( 2^{j} \lVert \Delta_{j} \widetilde{B} \rVert_{L^{2}} \sum_{m \leq j-1} 2^{\frac{d}{2}m} \lVert \Delta_{m} B^{(1)} \rVert_{L^{2}}\\
        &+\lVert \Delta_{j} B^{(1)} \rVert_{L^{2}} \sum_{m \leq j-1} 2^{(1+\frac{d}{2})m}\lVert \Delta_{m} \widetilde{B} \rVert_{L^{2}}\\
        &+2^{(1+\frac{d}{2})j} \sum_{k \geq j-4} \lVert \Delta_{k} B^{(1)} \rVert_{L^{2}}\lVert \widetilde{\Delta}_{k} \widetilde{B} \rVert_{L^{2}}\bigg). 
    \end{aligned}
\end{equation*}
We have
\begin{equation*}
    \begin{aligned}
        \lVert \Delta_{j}(\widetilde{u}\cdot \nabla u^{(2)})\rVert_{L^{2}}
        &\leq C\bigg(2^{j}\lVert\Delta_{j}u^{(2)} \rVert_{L^{2}} \sum_{m \leq j-1 } 2^{\frac{d}{2}m} \lVert \Delta_{m} \widetilde{u}\rVert_{L^{2}}\\
        &+ \lVert \Delta_{j} \widetilde{u}\rVert_{L^{2}} \sum_{m \leq j-1} 2^{(1+\frac{d}{2})m} \lVert \Delta_{m} u^{(2)} \rVert_{L^{2}}\\
        &+ 2^{(1+\frac{d}{2})j} \sum_{k \geq j-4} \lVert \Delta_{k} \widetilde{u} \rVert_{L^{2}} \lVert \widetilde{\Delta}_{k} u^{(2)}\rVert_{L^{2}}\bigg)
    \end{aligned}
\end{equation*}
and
\begin{equation*}
    \begin{aligned}
        \lVert \Delta_{j}(\widetilde{B}\cdot \nabla B^{(2)})\rVert_{L^{2}}
        &\leq C\bigg(2^{j} \lVert \Delta_{j} B^{(2)} \rVert_{L^{2}} \sum_{m \leq j-1 } 2^{\frac{d}{2}m} \lVert \Delta_{m} \widetilde{B}\rVert_{L^{2}}\\
        &+ \lVert \Delta_{j} \widetilde{B}\rVert_{L^{2}} \sum_{m \leq j-1} 2^{(1+\frac{d}{2})m} \lVert \Delta_{m} B^{(2)} \rVert_{L^{2}}\\
        &+ 2^{(1+\frac{d}{2})j} \sum_{k \geq j-4} \lVert \Delta_{k} \widetilde{B} \rVert_{L^{2}} \lVert \widetilde{\Delta}_{k} B^{(2)}\rVert_{L^{2}}\bigg).
    \end{aligned}
\end{equation*}
Considering the $L^{p}$-norm in time for $1\leq p \leq \infty $ and Young's inequality for time convolution, we derive\begin{equation}\label{eq74}
    \begin{aligned}
         \lVert\Delta_{j} \widetilde{u}||_{L_{t}^{p}(L^{2})} &\leq
    C \lVert e^{-c_{0}2^{2 \alpha j}t}\rVert_{L^{p}}\bigg(\lVert E_{1}\rVert_{L_{t}^{1}}+ \lVert E_{2} \rVert_{L_{t}^{1}}+ \lVert \Delta_{j}(\widetilde{u} \cdot \nabla u^{(2)} \rVert_{L_{t}^{1}(L^{2})}\\
    &+ \lVert \Delta_{j}(\widetilde{B} \cdot \nabla B^{(2)} \rVert_{L_{t}^{1}(L^{2})} 
    + \lVert \Delta_{j} \widetilde{\theta}\rVert_{L_{t}^{1}(L^{2})}\bigg)
    \end{aligned}
\end{equation}
Multiplying $(\ref{eq74})$ by $2^{(\frac{d}{2}-2\alpha+\frac{2\alpha}{p})j}$ and taking the supremum with respect to $j$,  we get
\begin{equation*}
    \begin{aligned}
         \lVert\Delta_{j} \widetilde{u}||_{L_{t}^{p}\big(B_{2,\infty}^{\frac{d}{2}-2\alpha+\frac{2\alpha}{p}}\big)} &\leq
    C \sup_{j} 2^{(\frac{d}{2}-2\alpha)}\bigg(\lVert E_{1}\rVert_{L_{t}^{1}}+ \lVert E_{2} \rVert_{L_{t}^{1}}+ \lVert \Delta_{j}(\widetilde{u} \cdot \nabla u^{(2)} \rVert_{L_{t}^{1}(L^{2})}\\
    &+ \lVert \Delta_{j}(\widetilde{B} \cdot \nabla B^{(2)} \rVert_{L_{t}^{1}(L^{2})} + \lVert \Delta_{j} \widetilde{\theta}\rVert_{L_{t}^{1}(L^{2})}\bigg).
    \end{aligned}
\end{equation*}
For $j=-1$, the procedure is equivalent to deriving (\ref{eq47}). Estimates (\ref{eq9})-(\ref{eq15}) give
\begin{equation*}
    \begin{aligned}
    \sup_{j} 2^{(\frac{d}{2}-2\alpha)j} \lVert E_{1} \rVert_{L_{t}^{1}}&\leq C \lVert \widetilde{u} \rVert_{L_{t}^{\infty}\big(B_{2,\infty}^{\frac{d}{2}-2\alpha} \big)}    \lVert u^{(1)} \rVert_{L_{t}^{1}\big(B_{2,1}^{1+\frac{d}{2}} \big)} , 
\\
    \sup_{j} 2^{(\frac{d}{2}-2\alpha)j} \lVert E_{2} \rVert_{L_{t}^{1}}&\leq C \lVert \widetilde{B} \rVert_{L_{t}^{\infty}\big(B_{2,\infty}^{\frac{d}{2}-2\alpha} \big)}    \lVert B^{(1)} \rVert_{L_{t}^{1}\big(B_{2,1}^{1+\frac{d}{2}} \big)},  
\\
    \sup_{j} 2^{(\frac{d}{2}-2\alpha)j} \lVert \Delta_{j}(\widetilde{u}\cdot \nabla u^{(2)})\rVert_{L_{t}^{1}(L^{2}) }&\leq C \lVert \widetilde{u} \rVert_{L_{t}^{\infty}\big(B_{2,\infty}^{\frac{d}{2}-2\alpha} \big)}    \lVert u^{(2)} \rVert_{L_{t}^{1}\big(B_{2,1}^{1+\frac{d}{2}} \big)},  
\\
    \sup_{j} 2^{(\frac{d}{2}-2\alpha)j} \lVert \Delta_{j}(\widetilde{B}\cdot \nabla B^{(2)})\rVert_{L_{t}^{1}(L^{2}) }&\leq C \lVert \widetilde{B} \rVert_{L_{t}^{\infty}\big(B_{2,\infty}^{\frac{d}{2}-2\alpha} \big)}    \lVert B^{(2)} \rVert_{L_{t}^{1}\big(B_{2,1}^{1+\frac{d}{2}} \big)}. 
    \end{aligned}
\end{equation*}
We can estimate the third term and the final term as follows:
\begin{equation*}
    \sup_{j} 2^{(\frac{d}{2}-2\alpha)j} \lVert \Delta_{j}\widetilde{\theta} \rVert_{L_{t}^{1}(L^{2})}  \leq   C \lVert \widetilde{\theta} \rVert_{L_{t}^{1}\big(B_{2,\infty}^{\frac{d}{2}-2\alpha} \big)}. 
\end{equation*}
Then, taking $p= \infty,  p = 1$ and $p=2 $ in (\ref{eq74}), we derive
\begin{equation}\label{eq75}
    \begin{aligned}
            \lVert \widetilde{u} \rVert_{\widetilde{L}_{t}^{\infty}(B_{2,\infty}^{\frac{d}{2}-2\alpha})}+\lVert \widetilde{u} \rVert_{\widetilde{L}_{t}^{1}(B_{2,\infty}^{\frac{d}{2}})}+\lVert \widetilde{u} \rVert_{\widetilde{L}_{t}^{2}(B_{2,\infty}^{\frac{d}{2}-\alpha})}
            &\leq C \bigg( \lVert \big(u^{(1)},u^{(2)}\big )\rVert_{L_{t}^{1}\big(B_{2,1}^{1+\frac{d}{2}} \big)}  \lVert \widetilde{u} \rVert_{\widetilde{L}_{t}^{\infty}(B_{2,\infty}^{\frac{d}{2}-2\alpha})}
    \\
    &+\lVert \big(B^{(1)},B^{(2)}\big )\rVert_{L_{t}^{1}\big(B_{2,1}^{1+\frac{d}{2}} \big)}  \lVert \widetilde{B} \rVert_{\widetilde{L}_{t}^{\infty}(B_{2,\infty}^{\frac{d}{2}-2\alpha})}\\
    &+  \lVert \widetilde{\theta} \rVert_{{L}_{t}^{1}(B_{2,\infty}^{\frac{d}{2}-2\alpha})}\bigg).
    \end{aligned}
\end{equation}
Applying $\Delta_{j}$ to (\ref{eq72})$_{2}$, and dotting the equation by $\Delta_{j} \widetilde{B}$, we get
\begin{equation*}
    \begin{aligned}
         \dfrac{1}{2} \dfrac{d}{dt} \lVert \Delta_{j} \widetilde{B}\rVert_{L^{2}}^{2} + \eta \lVert \Lambda^{\alpha} \Delta_{j} \widetilde{B} \rVert_{L^{2}}^{2}
         &= - \int_{\mathbb{R^{d}}} \Delta_{j}(u^{(1)} \cdot \nabla \widetilde{B})\cdot \Delta_{j} \widetilde{B} dx - \int_{\mathbb{R^{d}}} \Delta_{j}(\widetilde{u} \cdot \nabla B^{(2)})\cdot \Delta_{j} \widetilde{B} dx
    \\ 
  &  + \int_{\mathbb{R^{d}}} \Delta_{j}(B^{(1)} \cdot \nabla \widetilde{u})\cdot \Delta_{j} \widetilde{B} dx +
    \int_{\mathbb{R^{d}}} \Delta_{j}(\widetilde{B} \cdot \nabla u^{(2)})\cdot \Delta_{j} \widetilde{B} dx.  
    \end{aligned}
\end{equation*}
Through analogous reasoning to deriving (\ref{eq73}), we conclude 
\begin{equation*}
       \lVert\Delta_{j} \widetilde{B}||_{L^{2}}^{2} \leq
    C \int_{0}^{t} e^{-c_{1}2^{2 \alpha j}(t-\tau)}\bigg(F_{1}+F_{2}+ \lVert \Delta_{j}(\widetilde{u} \cdot \nabla B^{(2)} )\rVert_{L^{2}}+ \lVert \Delta_{j}(\widetilde{B} \cdot \nabla u^{(2)} \rVert_{L^{2}} \bigg),
\end{equation*}
where
\begin{equation*}
    \begin{aligned}
        F_{1} &:= C\Bigg( \lVert \Delta_{j} u^{(1)} \rVert_{L^{2}} \sum_{m \leq j-1} 2^{(1+\frac{d}{2})m} \lVert \Delta_{m} \widetilde{B} \rVert_{L^{2}}\\
        &+\lVert \Delta_{j} \widetilde{B} \rVert_{L^{2}} \sum_{m \leq j-1} 2^{(1+\frac{d}{2})m}\lVert \Delta_{m} u^{(1)} \rVert_{L^{2}}\\
        &+2^{(1+\frac{d}{2})j} \sum_{k \geq j-4} \lVert \Delta_{k} u^{(1)} \rVert_{L^{2}}\lVert \widetilde{\Delta}_{k} \widetilde{B} \rVert_{L^{2}}\Bigg)
    \end{aligned}
\end{equation*}
and
\begin{equation*}
    \begin{aligned}
        F_{2} &= C\Bigg( 2^{j} \lVert \Delta_{j} \widetilde{u} \rVert_{L^{2}} \sum_{m \leq j-1} 2^{\frac{d}{2}m} \lVert \Delta_{m} B^{(1)} \rVert_{L^{2}}\\
        &+\lVert \Delta_{j} B^{(1)} \rVert_{L^{2}} \sum_{m \leq j-1} 2^{(1+\frac{d}{2})m}\lVert \Delta_{m} \widetilde{u} \rVert_{L^{2}}\\
        &+2^{(1+\frac{d}{2})j} \sum_{k \geq j-4} \lVert \Delta_{k} B^{(1)} \rVert_{L^{2}}\lVert \widetilde{\Delta}_{k} \widetilde{u} \rVert_{L^{2}}\Bigg).
    \end{aligned}
\end{equation*}
So,
\begin{equation*}
    \begin{aligned}
        \lVert \Delta_{j}(\widetilde{u}\cdot \nabla B^{(2)})\rVert_{L^{2}}
        &\leq C\bigg(2^{j}\lVert\Delta_{j}B^{(2)} \rVert_{L^{2}} \sum_{m \leq j-1 } 2^{\frac{d}{2}m} \lVert \Delta_{m} \widetilde{u}\rVert_{L^{2}}\\
        &+ \lVert \Delta_{j} \widetilde{u}\rVert_{L^{2}} \sum_{m \leq j-1} 2^{(1+\frac{d}{2})m} \lVert \Delta_{m} B^{(2)} \rVert_{L^{2}}\\
        & + 2^{(1+\frac{d}{2})j} \sum_{k \geq j-4} \lVert \Delta_{k} \widetilde{u} \rVert_{L^{2}} \lVert \widetilde{\Delta}_{k} B^{(2)}\rVert_{L^{2}}\bigg).
    \end{aligned}
\end{equation*}
Also,
\begin{equation*}
    \begin{aligned}
        \lVert \Delta_{j}(\widetilde{B}\cdot \nabla u^{(2)})\rVert_{L^{2}}
        &\leq C\bigg(2^{j} \lVert \Delta_{j} u^{(2)} \rVert_{L^{2}} \sum_{m \leq j-1 } 2^{\frac{d}{2}m} \lVert \Delta_{m} \widetilde{B}\rVert_{L^{2}}\\
        &+ \lVert \Delta_{j} \widetilde{B}\rVert_{L^{2}} \sum_{m \leq j-1} 2^{(1+\frac{d}{2})m} \lVert \Delta_{m} u^{(2)} \rVert_{L^{2}}\\
        &+ 2^{(1+\frac{d}{2})j} \sum_{k \geq j-4} \lVert \Delta_{k} \widetilde{B} \rVert_{L^{2}} \lVert \widetilde{\Delta}_{k} u^{(2)}\rVert_{L^{2}}\bigg)
    \end{aligned}
\end{equation*}
Taking the $L^{p}$-norm in time for $1\leq p \leq \infty $ and applying Young’s inequality for the time convolution, we yield
\begin{equation}\label{eq76}
        \lVert\Delta_{j} \widetilde{B}||_{L_{t}^{p}(L^{2})} \leq
    C \lVert e^{-c_{1}2^{2 \alpha j}t}\rVert_{L^{p}}\bigg(\lVert F_{1}\rVert_{L_{t}^{1}}+ \lVert F_{2} \rVert_{L_{t}^{1}}+ \lVert \Delta_{j}(\widetilde{u} \cdot \nabla B^{(2)} \rVert_{L_{t}^{1}(L^{2})}+ \lVert \Delta_{j}(\widetilde{B} \cdot \nabla u^{(2)} \rVert_{L_{t}^{1}(L^{2})} \bigg).
\end{equation}
Multiplying (\ref{eq76})by $2^{(\frac{d}{2}-2\alpha+\frac{2\alpha}{p})j}$ and taking the supremum with respect to $j$ for $j=-1$ the method is same as deriving (\ref{eq47}), we get
\begin{equation}\label{eq77}
      \lVert \widetilde{B}||_{L_{t}^{p}\big(B_{2,\infty}^{\frac{d}{2}-2\alpha+\frac{2\alpha}{p}}\big)} \leq
    C \sup_{j} 2^{(\frac{d}{2}-2\alpha)j}\bigg(\lVert F_{1}\rVert_{L_{t}^{1}}+ \lVert F_{2} \rVert_{L_{t}^{1}}+ \lVert \Delta_{j}(\widetilde{u} \cdot \nabla B^{(2)} \rVert_{L_{t}^{1}(L^{2})}+ \lVert \Delta_{j}(\widetilde{B} \cdot \nabla u^{(2)} \rVert_{L_{t}^{1}(L^{2})}\bigg ) 
\end{equation}
The estimates (\ref{eq23})–(\ref{eq28}) illustrate that:
\begin{equation*}
    \begin{aligned}
   \sup_{j} 2^{(\frac{d}{2}-2\alpha)j} \lVert F_{1} \rVert_{L_{t}^{1}}&\leq C \lVert \widetilde{B} \rVert_{L_{t}^{\infty}\big(B_{2,\infty}^{\frac{d}{2}-2\alpha} \big)}    \lVert u^{(1)} \rVert_{L_{t}^{1}\big(B_{2,1}^{1+\frac{d}{2}} \big)},  
\\
    \sup_{j} 2^{(\frac{d}{2}-2\alpha)j} \lVert F_{2} \rVert_{L_{t}^{1}}&\leq C \lVert \widetilde{u} \rVert_{L_{t}^{\infty}\big(B_{2,\infty}^{\frac{d}{2}-2\alpha} \big)}    \lVert B^{(1)} \rVert_{L_{t}^{1}\big(B_{2,1}^{1+\frac{d}{2}} \big)},  
\\
    \sup_{j} 2^{(\frac{d}{2}-2\alpha)j} \lVert \Delta_{j}(\widetilde{u}\cdot \nabla B^{(2)})\rVert_{L_{t}^{1}(L^{2}) }&\leq C \lVert \widetilde{u} \rVert_{L_{t}^{\infty}\big(B_{2,\infty}^{\frac{d}{2}-2\alpha} \big)}    \lVert B^{(2)} \rVert_{L_{t}^{1}\big(B_{2,1}^{1+\frac{d}{2}} \big),}  
\\
    \sup_{j} 2^{(\frac{d}{2}-2\alpha)j} \lVert \Delta_{j}(\widetilde{B}\cdot \nabla u^{(2)})\rVert_{L_{t}^{1}(L^{2}) }&\leq C \lVert \widetilde{B} \rVert_{L_{t}^{\infty}\big(B_{2,\infty}^{\frac{d}{2}-2\alpha} \big)}    \lVert u^{(2)} \rVert_{L_{t}^{1}\big(B_{2,1}^{1+\frac{d}{2}} \big)} . 
    \end{aligned}
\end{equation*}
Taking $p= \infty,  p = 1$ and $p=2 $ in (\ref{eq77}), we obtain
\begin{equation*}
    \begin{aligned}
           \lVert \widetilde{B} \rVert_{\widetilde{L}_{t}^{\infty}(B_{2,\infty}^{\frac{d}{2}-2\alpha})}+\lVert \widetilde{B} \rVert_{\widetilde{L}_{t}^{1}(B_{2,\infty}^{\frac{d}{2}})}+\lVert \widetilde{B} \rVert_{\widetilde{L}_{t}^{2}(B_{2,\infty}^{\frac{d}{2}-\alpha})}
           &\leq C \bigg( \lVert \big(B^{(1)},B^{(2)}\big )\rVert_{L_{t}^{1}\big(B_{2,1}^{1+\frac{d}{2}} \big)}  \lVert \widetilde{u} \rVert_{\widetilde{L}_{t}^{\infty}(B_{2,\infty}^{\frac{d}{2}-2\alpha})}\\
    &+\lVert \big(u^{(1)},u^{(2)}\big )\rVert_{L_{t}^{1}\big(B_{2,1}^{1+\frac{d}{2}} \big)}  \lVert \widetilde{B} \rVert_{\widetilde{L}_{t}^{\infty}(B_{2,\infty}^{\frac{d}{2}-2\alpha})}\bigg),
    \end{aligned}
\end{equation*}
\begin{equation}\label{eq78}
        C \lVert (u^{(1)},u^{(2)},B^{(1)},B^{(2)})\rVert_{L_{t}^{1}(B_{2,1}^{1+\frac{d}{2}})} \leq \dfrac{1}{4}.
\end{equation}
Thus, from (\ref{eq75}), (\ref{eq77}) and (\ref{eq78}), we have
\begin{equation}\label{eq79}
       \lVert \widetilde{B} \rVert_{\widetilde{L}_{t}^{1}(B_{2,\infty}^{\frac{d}{2}})}+ \lVert \widetilde{u} \rVert_{\widetilde{L}_{t}^{1}(B_{2,\infty}^{\frac{d}{2}})}+\lVert \widetilde{u} \rVert_{\widetilde{L}_{t}^{2}(B_{2,\infty}^{\frac{d}{2}-\alpha})}\leq c \lVert \widetilde{\theta} \rVert_{{L}_{t}^{1}(B_{2,\infty}^{\frac{d}{2}-2\alpha})}.
\end{equation}
Now, we estimate $\lVert \widetilde{\theta} \rVert_{{L}_{t}^{1}(B_{2,\infty}^{\frac{d}{2}-2\alpha})}.$ Applying $ \Delta_{j}$ to (\ref{eq72})$_{3}$
and dotting the equation by $\Delta_{j}\widetilde{\theta}$, we derive
\begin{equation*}
    \dfrac{1}{2} \dfrac{d}{dt} \lVert \Delta_{j} \widetilde{\theta}\rVert_{L^{2}}^{2} = - \int_{\mathbb{R^{d}}} \Delta_{j}(u^{(1)} \cdot \nabla \widetilde{\theta})\cdot \Delta_{j} \widetilde{\theta} dx - \int_{\mathbb{R^{d}}} \Delta_{j}(\widetilde{u} \cdot \nabla \theta^{(2)})\cdot \Delta_{j} \widetilde{\theta} dx
    + \int_{\mathbb{R^{d}}} \Delta_{j}(\widetilde{u} \cdot e_{d})\cdot \Delta_{j} \widetilde{\theta} dx. 
\end{equation*}
After subtracting $\lVert \Delta_{j} \widetilde{\theta}\rVert_{L^{2}}^{2} $ from both sides of the inequality and integrating in time in accordance with Lemma (\ref{lem2}), we acquire
 \begin{equation}\label{eq80}
       \lVert \Delta_{j} \widetilde{\theta} \rVert_{L^{2}} \leq C \int_{0}^{t} \big( G + \lVert \Delta_{j}(\widetilde{u}\cdot \nabla \theta^{(2)})\rVert_{L^{2}}+ \lVert \Delta_{j} \widetilde{u} \rVert_{L^{2}}\big) d\tau
 \end{equation}
where
\begin{equation*}
    \begin{aligned}
          G
          &:= C\bigg( \lVert \Delta_{j} u^{(1)} \rVert_{L^{2}} \sum_{m \leq j-1} 2^{(1+\frac{d}{2})m} \lVert \Delta_{m} \widetilde{\theta} \rVert_{L^{2}}\\
          &+\lVert \Delta_{j} \widetilde{\theta} \rVert_{L^{2}} \sum_{m \leq j-1} 2^{(1+\frac{d}{2})m}\lVert \Delta_{m} u^{(1)} \rVert_{L^{2}}\\
          &+2^{(1+\frac{d}{2})j} \sum_{k \geq j-4} \lVert \Delta_{k} u^{(1)} \rVert_{L^{2}}\lVert \widetilde{\Delta}_{k} \widetilde{\theta} \rVert_{L^{2}}\bigg). 
    \end{aligned}
\end{equation*}
So,
\begin{equation*}
    \begin{aligned}
        \lVert \Delta_{j}(\widetilde{u}\cdot \nabla \theta^{(2)})\rVert_{L^{2}}
        &\leq C\bigg(2^{j}\lVert\Delta_{j}\theta^{(2)} \rVert_{L^{2}} \sum_{m \leq j-1 } 2^{\frac{d}{2}m} \lVert \Delta_{m} \widetilde{u}\rVert_{L^{2}}\\
        &+ \lVert \Delta_{j} \widetilde{u}\rVert_{L^{2}} \sum_{m \leq j-1} 2^{(1+\frac{d}{2})m} \lVert \Delta_{m} \theta^{(2)} \rVert_{L^{2}}\\
        &+ 2^{(1+\frac{d}{2})j} \sum_{k \geq j-4} \lVert \Delta_{k} \widetilde{u} \rVert_{L^{2}} \lVert \widetilde{\Delta}_{k} \theta^{(2)}\rVert_{L^{2}}\bigg)
    \end{aligned}
\end{equation*}
Multiplying (\ref{eq80}) by $2^{(\frac{d}{2}-2\alpha)j}$ and taking the supremum with respect to $j$, we obtain
\begin{equation}\label{eq81}
        \lVert \widetilde{\theta}\rVert_{B_{2,\infty}^{\frac{d}{2}-2\alpha} }\leq
    C \sup_{j} 2^{(\frac{d}{2}-2\alpha)j}\bigg(\lVert G\rVert_{L_{t}^{1}}+ \lVert \Delta_{j}(\widetilde{u} \cdot \nabla \theta^{(2)} \rVert_{L_{t}^{1}(L^{2})}+ \lVert \Delta_{j} \widetilde{u}  \rVert_{L_{t}^{1}(L^{2})}\bigg) 
\end{equation}
It is possible to estimate the three terms on the right as:
\begin{equation*}
    \begin{aligned}
    \sup_{j} 2^{(\frac{d}{2}-2\alpha)j} \lVert G \rVert_{L_{t}^{1}}&\leq C \lVert \widetilde{\theta} \rVert_{L_{t}^{\infty}\big(B_{2,\infty}^{\frac{d}{2}-2\alpha} \big)}    \lVert u^{(1)} \rVert_{L_{t}^{1}\big(B_{2,1}^{1+\frac{d}{2}} \big)},  
\\
    \sup_{j} 2^{(\frac{d}{2}-2\alpha)j} \lVert \Delta_{j}(\widetilde{u}\cdot \nabla \theta^{(2)})\rVert_{L_{t}^{1}(L^{2}) }&\leq C \sqrt{T}\lVert \widetilde{u} \rVert_{\widetilde{L}_{t}^{2}\big(B_{2,\infty}^{\frac{d}{2}-2\alpha} \big)}    \lVert \theta^{(2)} \rVert_{\widetilde{L}_{t}^{\infty}\big(B_{2,1}^{1+\frac{d}{2}-\alpha} \big)},  
\\
    \sup_{j} 2^{(\frac{d}{2}-2\alpha)j} \lVert \Delta_{j} \widetilde{u} \rVert_{L_{t}^{1}(L^{2}) }&\leq C   \lVert \widetilde{u} \rVert_{\widetilde{L}_{t}^{1}\big(B_{2,\infty}^{\frac{d}{2}} \big)}.  
    \end{aligned}
\end{equation*}
Inserting them into (\ref{eq81}), gives us
\begin{equation}\label{eq82}
        \lVert \widetilde{\theta}\rVert_{B_{2,\infty}^{\frac{d}{2}-2\alpha} }\leq  C \sqrt{T}\lVert \widetilde{u} \rVert_{\widetilde{L}_{t}^{2}\big(B_{2,\infty}^{\frac{d}{2}-2\alpha} \big)}    \lVert \theta^{(2)} \rVert_{\widetilde{L}_{t}^{\infty}\big(B_{2,1}^{1+\frac{d}{2}-\alpha}\big)}+C   \lVert \widetilde{u} \rVert_{\widetilde{L}_{t}^{1}\big(B_{2,\infty}^{\frac{d}{2}} \big)}.  
\end{equation}
We employ the Chemin-Lerner type Besov space and Lemma (\ref{lem4}) combining (\ref{eq79}) and (\ref{eq82}). The inequality of Gr\"onwall derives
\begin{equation*}
    \lVert \widetilde{B} \rVert_{\widetilde{L}_{t}^{1}(B_{2,\infty}^{\frac{d}{2}})}= \lVert \widetilde{u} \rVert_{\widetilde{L}_{t}^{1}(B_{2,\infty}^{\frac{d}{2}})}=\lVert \widetilde{u} \rVert_{\widetilde{L}_{t}^{2}(B_{2,\infty}^{\frac{d}{2}-\alpha})}=0.
\end{equation*}
Consequently, 
$$
\lVert \widetilde{\theta} \rVert_{{L}_{t}^{1}(B_{2,\infty}^{\frac{d}{2}})} =0\;\;\mbox{for all}\;\; t\in [0,T].
$$
\end{proof}
\section{Conclusion}
In this study, we investigated the existence and uniqueness of local weak solutions for the d-dimensional $(d \geq 2)$ fractional magnetic B\'enard system without thermal diffusion. Notably, although kinematic viscosity and magnetic diffusion are present, thermal diffusion is zero. The thought invites us to think about more issues concerning these constraints. Four open problems that we think may be important will be introduced:
\begin{equation}\label{sys3}
\left\{
\begin{aligned}
    & \partial_{t} u + u\cdot\nabla u  =- \mu (-\Delta)^{\alpha}u-\nabla p + B\cdot \nabla B +\theta e_{d}, \quad (x,t)\in \mathbb{R}^{d}\times (0,\infty),
     \\
    & \partial_{t} B + u\cdot\nabla B =- \eta (-\Delta)^{\beta}B + B\cdot \nabla u,\quad (x,t)\in \mathbb{R}^{d}\times (0,\infty),
     \\
   & \partial_{t}\theta + u\cdot\nabla\theta = u\cdot e_{d},\quad (x,t)\in \mathbb{R}^{d}\times (0,\infty),
  \\
&\Div{u}=0 ~~~~~\Div{B}=0\quad\quad (x,t)\in \mathbb{R}^{d}\times (0,\infty),
\\
& (u,B,\theta )\mid_{t=0} = (u_{0}, B_{0},\theta _{0}),\quad (x,t)\in \mathbb{R}^{d}\times (0,\infty).
\end{aligned}
\right.
\end{equation}
\section{Further questions}
We draw the reader's attention to the following outstanding issues as we wrap up our analysis. In any event, we are prepared to work together.\\
\textbf{Open problem 01:} Existence and Uniqueness of local weak solution of d-dimensional  Magnetic-Bénard system with  fractional dissipation in Besov Space
\begin{equation}
\left\{
\begin{aligned}
    & \partial_{t} u + u\cdot\nabla u  =- \mu (-\Delta)^{\alpha}u-\nabla p + B\cdot \nabla B +\theta e_{d},\quad (x,t)\in \mathbb{R}^{d}\times (0,\infty),
     \\
    & \partial_{t} B + u\cdot\nabla B =- \eta (-\Delta)^{\beta}B + B\cdot \nabla u,\quad (x,t)\in \mathbb{R}^{d}\times (0,\infty),
     \\
   & \partial_{t}\theta + u\cdot\nabla\theta =-\kappa(-\Delta)^{\gamma}\theta + u\cdot e_{d},\quad (x,t)\in \mathbb{R}^{d}\times (0,\infty),
  \\
&\Div{u}=0 ~~~~~\Div{B}=0,\quad (x,t)\in \mathbb{R}^{d}\times (0,\infty),
\\
& (u,B,\theta )\mid_{t=0} = (u_{0}, B_{0},\theta _{0}),\quad (x,t)\in \mathbb{R}^{d}\times (0,\infty).
\end{aligned}
\right.
\end{equation}
   It is noteworthy that every parameter is present in the first open problem. This requires determining if a local weak solution exists and is unique in Besov space, based on the parameters and outcomes that have been altered from Theorem \ref{Th1}.\\
	\textbf{Open problem 02:} Existence and uniqueness of  the non-resistive  Magnetic-Bénard system With  fractional dissipation in Besov Space
\begin{equation}
\left\{
\begin{aligned}
    & \partial_{t} u + u\cdot\nabla u  =- \mu (-\Delta)^{\alpha}u-\nabla p + B\cdot \nabla B +\theta e_{d},\quad (x,t)\in \mathbb{R}^{d}\times (0,\infty),
     \\
    & \partial_{t} B + u\cdot\nabla B =
    B\cdot \nabla u  \quad\quad (x,t)\in \mathbb{R}^{d}\times (0,\infty)
     \\
   & \partial_{t}\theta + u\cdot\nabla\theta =-\kappa(-\Delta)^{\gamma}\theta + u\cdot e_{d},\quad (x,t)\in \mathbb{R}^{d}\times (0,\infty),
  \\
&\Div{u}=0 ~~~~~\Div{B}=0,\quad (x,t)\in \mathbb{R}^{d}\times (0,\infty),
\\
& (u,B,\theta )\mid_{t=0} = (u_{0}, B_{0},\theta _{0}),\quad (x,t)\in \mathbb{R}^{d}\times (0,\infty).
\end{aligned}
\right.
\end{equation}
We treat magnetic diffusion in the second open problem as zero, i.e., $\eta=0$, and assume the presence of both thermal and kinetic viscosity. This requires determining if a local weak solution exists and is unique in Besov space, based on the parameters and outcomes that have been altered from Theorem \ref{Th1}.\\
\textbf{Open problem 03:} Existence and uniqueness of local weak solution of d-dimensional Magnetic-B\'enard
system With  kinematic viscosity zero dissipation in Besov Space
\begin{equation}
\left\{
\begin{aligned}
    & \partial_{t} u + u\cdot\nabla u  =\nabla p + B\cdot \nabla B +\theta e_{d},\quad (x,t)\in \mathbb{R}^{d}\times (0,\infty),
     \\
   & \partial_{t} B + u\cdot\nabla B =- \eta (-\Delta)^{\beta}B + B\cdot \nabla u,\quad (x,t)\in \mathbb{R}^{d}\times (0,\infty),
     \\
   & \partial_{t}\theta + u\cdot\nabla\theta =-\kappa(-\Delta)^{\gamma}\theta + u\cdot e_{d},\quad (x,t)\in \mathbb{R}^{d}\times (0,\infty),
  \\
&\Div{u}=0 ~~~~~\Div{B}=0,\quad (x,t)\in \mathbb{R}^{d}\times (0,\infty),
\\
& (u,B,\theta )\mid_{t=0} = (u_{0}, B_{0},\theta _{0}),\quad (x,t)\in \mathbb{R}^{d}\times (0,\infty).
\end{aligned}
\right.
\end{equation}
   In open problem three, we consider kinematic viscosity to be zero, as are magnetic diffusion and thermal diffusion. This necessitates identifying the existence and uniqueness of a weak local solution. in Besov space, according to the conditions and results modified from Theorem \ref{Th1}.\\
\textbf{Open problem 04:} Existence and uniqueness of local weak solution of d-Ddmensional Magnetic-B\'enard
system with  kinematic viscosity zero dissipation and without thermal diffusion in Besov Space
\begin{equation}
\left\{
\begin{aligned}
    & \partial_{t} u + u\cdot\nabla u  =\nabla p + B\cdot \nabla B +\theta e_{d},\quad (x,t)\in \mathbb{R}^{d}\times (0,\infty),
     \\
   & \partial_{t} B + u\cdot\nabla B =- \eta (-\Delta)^{\beta}B + B\cdot \nabla u,\quad (x,t)\in \mathbb{R}^{d}\times (0,\infty),
     \\
   & \partial_{t}\theta + u\cdot\nabla\theta = u\cdot e_{d},\quad (x,t)\in \mathbb{R}^{d}\times (0,\infty),
  \\
&\Div{u}=0 ~~~~~\Div{B}=0,\quad (x,t)\in \mathbb{R}^{d}\times (0,\infty),
\\
& (u,B,\theta )\mid_{t=0} = (u_{0}, B_{0},\theta _{0}),\quad (x,t)\in \mathbb{R}^{d}\times (0,\infty).
\end{aligned}
\right.
\end{equation}
  In the final open problem, we analyze kinetic viscosity, and while thermal diffusion is zero, magnetic diffusion is present. This requires determining the presence and uniqueness of a local weak solution. The criteria and findings of Theorem \ref{Th1} are adjusted in Besov space.

\section*{Acknowledgements}
The authors thank the anonymous referees for their constructive criticism and suggestions.

\end{document}